\documentclass[12pt]{amsart}
\usepackage{enumerate}
\usepackage{mathdots,bm,url}
\usepackage{amsfonts,amssymb,mathscinet,a4wide}
\usepackage[all]{xy}
\usepackage[stable]{footmisc}

\author{Erez Lapid}
\address{Department of Mathematics, Weizmann Institute of Science, Rehovot 76100 Israel}
\email{erez.m.lapid@gmail.com}
\author{Zhengyu Mao}
\address{Department of Mathematics and Computer Science, Rutgers University, Newark, NJ 07102, USA}
\email{zmao@rutgers.edu}
\title[Whittaker--Fourier coefficients on $\Mp_{n}$]{Whittaker--Fourier coefficients of cusp forms on $\Mp_{n}$: reduction to a local statement}
\date{\today}
\thanks{Authors partially supported by U.S.--Israel Binational Science Foundation Grant \# 057/2008}
\thanks{Second named author partially supported by NSF grants DMS 1000636 and 1400063 and by a fellowship from the Simons Foundation}

\newcommand{\A}{\mathbb{A}}                            
\newcommand{\C}{\mathbb{C}}                            
\newcommand{\R}{\mathbb{R}}                            
\newcommand{\Z}{\mathbb{Z}}                            
\newcommand{\AF}{\mathcal{A}}                          
\newcommand{\bs}{\backslash}
\newcommand{\Hom}{\operatorname{Hom}}
\newcommand{\vol}{\operatorname{vol}}
\newcommand{\cusp}{\operatorname{cusp}}
\newcommand{\temp}{\operatorname{temp}}
\newcommand{\mcusp}{\operatorname{mcusp}}
\newcommand{\dsum}{\oplus}
\newcommand{\iii}{\mathrm{i}}
\newcommand{\sgn}{\operatorname{sgn}}
\newcommand{\OO}{\mathcal{O}}                         
\newcommand{\eps}{\epsilon}

\newcommand{\K}{\mathbf{K}}                             
\newcommand{\stint}{\int^{st}}                         
\newcommand{\eisen}{\mathcal{E}}                       
\newcommand{\reseisen}{\eisen_{-k}}                    
\newcommand{\resM}{M_{-k}}                             
\newcommand{\resm}{m_{-k}}                             
\newcommand{\desc}{\mathcal{D}}
\newcommand{\des}{\mathcal{D}_\psi}                     
\newcommand{\desinv}{\mathcal{D}_{\psi^{-1}}}
\newcommand{\whitform}{A^{\psi}}                              
\newcommand{\whitformd}{A^{\psi^{-1}}}
\newcommand{\Mint}{A_e^{\psi}}                                
\newcommand{\Mintd}{A_e^{\psi^{-1}}}
\newcommand{\F}{\mathbb{F}}

\newcommand{\wgt}[1]{\nu(#1)}
\newcommand{\good}{good}

\renewcommand{\d}[1]{#1^{\vee}}
\newcommand{\alt}[1]{#1^{\wedge}}
\newcommand{\psiold}{\hat{\psi}}            
\newcommand{\animg}[1]{\widetilde{#1}}                        
\newcommand{\Mp}{\widetilde{\operatorname{Sp}}}
\newcommand{\Mat}{\operatorname{Mat}}

\newcommand{\GL}{\operatorname{GL}}
\newcommand{\Sp}{\operatorname{Sp}}
\newcommand{\SL}{\operatorname{SL}}
\newcommand{\SO}{\operatorname{SO}}
\newcommand{\sym}{\operatorname{sym}}
\newcommand{\Ind}{\operatorname{Ind}}                   
\newcommand{\res}{\operatorname{res}}                   
\newcommand{\Shal}{\mathfrak{S}}
\newcommand{\Levi}{M}
\newcommand{\GLnn}{{\mathbb M}}
\newcommand{\GLn}{{{\mathbb M}'}}

\newcommand{\Irr}{\operatorname{Irr}}                  
\newcommand{\rest}{\big|}                              
\newcommand{\Cusp}{\operatorname{Cusp}}                
\newcommand{\Mcusp}{\operatorname{MCusp}}              
\newcommand{\meta}{\operatorname{meta}}                
\newcommand{\shal}{\operatorname{Shal}}                
\newcommand{\gen}{\operatorname{gen}}                  
\newcommand{\sqr}{\operatorname{sqr}}                  
\newcommand{\PGL}{\operatorname{PGL}}

\newcommand{\FJc}{\operatorname{FJ}}                   
\newcommand{\FJ}{\operatorname{FJ}_{\psi_{N_\GLnn}}}   
\newcommand{\levi}{\varrho}                            
\newcommand{\toG}{\eta}                                
\newcommand{\toM}{\eta_{\GLnn}}                        
\newcommand{\toLevi}{\eta_{\Levi}}                     
\newcommand{\toMd}{\eta_{\GLnn}^\vee}                  
\newcommand{\toLevid}{\eta_{\Levi}^\vee}               
\newcommand{\toU}{\ell}
\newcommand{\toUbar}{\overline{\ell}}
\newcommand{\startran}[1]{\breve{#1}}             
\newcommand{\new}[1]{#1^\circ}                   
\newcommand{\swrz}{\mathcal{S}}
\newcommand{\modulus}{\delta}
\newcommand{\diag}{\operatorname{diag}}

\newcommand{\Heise}{\mathcal{H}}        
\newcommand{\whit}{\mathcal{W}}                          
\newcommand{\Whit}{\mathbb{W}}                           
\newcommand{\WhitM}{\Whit^{\psi_{N_\GLnn}}}                  
\newcommand{\WhitMd}{\Whit^{\psi_{N_\GLnn}^{-1}}}
\newcommand{\WhitML}{\Whit^{\psi_{N_\Levi}}}                  
\newcommand{\WhitMLd}{\Whit^{\psi_{N_\Levi}^{-1}}}
\newcommand{\WhitG}{\Whit^{\psi_{\tilde N}}}                  
\newcommand{\WhitGd}{\Whit^{\psi_{\tilde N}^{-1}}}
\newcommand{\symspace}{\mathfrak{s}}                 
\newcommand{\psiweil}{\psi_{\circ}}
\newcommand{\fel}{\epsilon_\pi}           
\newcommand{\one}{\epsilon}        
\newcommand{\wnn}{{w_0^\GLnn}}
\newcommand{\tr}{\operatorname{tr}}                        
\newcommand{\Hei}[1]{{#1}_{\Heise}}
\newcommand{\minors}{\mu}                    

\newcommand{\sprod}[2]{\left\langle#1,#2\right\rangle}
\newcommand{\abs}[1]{\left|{#1}\right|}
\newcommand{\norm}[1]{\lVert#1\rVert}
\newcommand{\sm}[4]{\left(\begin{smallmatrix}{#1}&{#2}\\{#3}&{#4}\end{smallmatrix}\right)}
\newcommand{\weil}{\omega}
\newcommand{\we}{\weil_{\psi}}                          
\newcommand{\weinv}{\weil_{\psi^{-1}}}                          
\newcommand{\wev}{\weil_{\psi_{N_\GLnn}}}
\newcommand{\wevinv}{\weil_{\psi_{N_\GLnn}^{-1}}}
\newcommand{\weilfctr}{\gamma_\psi}                       
\newcommand{\der}{\operatorname{der}}                     

\newcommand{\spclt}{\mathfrak{t}}
\newcommand{\expo}{\operatorname{e}}

\newtheorem{theorem}{Theorem}[section]
\newtheorem{lemma}[theorem]{Lemma}
\newtheorem{proposition}[theorem]{Proposition}
\newtheorem{remark}[theorem]{Remark}
\newtheorem{conjecture}[theorem]{Conjecture}
\newtheorem{definition}[theorem]{Definition}

\newtheorem{claim}[theorem]{``Claim''}

\numberwithin{equation}{section}

\begin{document}

\begin{abstract}
In a previous paper we formulated an analogue of the Ichino--Ikeda conjectures for Whittaker--Fourier coefficients of cusp forms
on quasi-split groups, as well as the metaplectic group of arbitrary rank.
In this paper we reduce the conjecture for the metaplectic group to a local conjectural identity.
We motivate this conjecture by giving a heuristic argument for the case $\widetilde{\SL}_2$.
In a subsequent paper we will prove the local identity in the $p$-adic case.
\end{abstract}

\maketitle

\setcounter{tocdepth}{1}
\tableofcontents

\section{Introduction}

In \cite{LMao5} we studied the Whittaker-Fourier coefficients of cusp forms on adelic quotients
of quasi-split groups and formulated a conjecture relating them to the Petersson inner product.
In general, the formulation relies on a strong form of Arthur's conjecture\footnote{In the function field case
a dramatic progress towards Arthur's conjecture in its strong form was made recently by V.~Lafforgue \cite{1209.5352}}
and is closely related
to recent conjectures of Sakellaridis--Venkatesh \cite{1203.0039} and slightly older ones of Ichino--Ikeda \cite{MR2585578},
both of which go back to the seminal work of Waldspurger \cite{MR783511, MR646366}.
In the case of (quasi-split) classical groups, as well as the metaplectic group (i.e.,
the metaplectic double cover of the symplectic group) we formulated the conjecture without appealing to Arthur's
conjecture (or work) by using the descent construction of Ginzburg--Rallis--Soudry \cite{MR2848523}
and the functorial transfer of generic representations of classical groups by Cogdell--Kim--Piatetski-Shapiro--Shahidi
\cite{MR1863734, MR2075885, MR2767514}.
We will recall the conjecture in the metaplectic case below. (See \cite{LMao5} for further discussion.)

The purpose of this paper is to reduce the global conjecture in the metaplectic case to a local conjectural identity.
The reduction is achieved, not surprisingly, by the descent method.
It is explained in \S\ref{sec: local to global} after we develop the prerequisites.
The analysis explains the power of two factor which appears in the formulation of the conjecture.
In a subsequent paper \cite{1404.2905} we prove the local identity in the $p$-adic case.
In yet another paper joint with Ichino \cite{1404.2909} we will directly relate this local identity in the square-integrable (generic)
case to the formal degree conjecture of Hiraga--Ichino--Ikeda \cite{MR2350057}.
Thereby, using Arthur's work and the work of Gan--Ichino \cite{MR3166215} we will establish the formal degree conjecture
for the metaplectic group as well as for the split odd orthogonal group.
We will also deduce the local identity in the archimedean case (for square-integrable representations) since the formal degree
conjecture is known in the archimedean case by the seminal work of Harish-Chandra \cite{MR0439994}.
In the current paper we will be content with a \emph{purely formal} argument (i.e., without worrying about convergence issues) for the case $n=1$
which is already quite non-trivial (see \S\ref{sec: n=1}).

Let us recall the conjecture of \cite{LMao5} in the case of the metaplectic group.
Let $F$ be a number field and $\A$ its adele ring.
Let $\Mp_n(\A)$ be the metaplectic double cover of $\Sp_n(\A)$ with the standard co-cycle.
We view $\Sp_n(F)$ as a subgroup of $\Mp_n(\A)$.
(All automorphic forms on $\Mp_n(\A)$ are implicity assumed to be genuine.)
Let $N'$ be the group of the upper unitriangular matrices in $\Sp_n$. We view $N'(\A)$ as a subgroup of $\Mp_n(\A)$.
(The co-cycle is trivial on $N'(\A)\times\Sp_n(\A)$.)
Fix a non-degenerate character $\psi_{\tilde N}$ on $N'(\A)$, trivial on $N'(F)$.
For a cusp form $\tilde\varphi$ of $\Sp_n(F)\bs\Mp_n(\A)$ we consider the Whittaker--Fourier coefficient
\[
\tilde\whit(\tilde\varphi)=\tilde\whit^{\psi_{\tilde N}}(\tilde\varphi):=
(\vol(N'(F)\bs N'(\A)))^{-1}\int_{N'(F)\bs N'(\A)}\tilde\varphi(n)\psi_{\tilde N}(n)^{-1}\ dn.
\]
If $\tilde\varphi^\vee$ is another cusp form on $\Sp_n(F)\bs\Mp_n(\A)$ we also set
\[
(\tilde\varphi,\tilde\varphi^\vee)_{\Sp_n(F)\bs\Sp_n(\A)}=(\vol(\Sp_n(F)\bs\Sp_n(\A)))^{-1}\int_{\Sp_n(F)\bs\Sp_n(\A)}\tilde\varphi(g)\tilde\varphi^\vee(g)\ dg.
\]

Given a finite set of places $S$ we defined in \cite{LMao5} a regularized integral
\[
\stint_{N'(F_S)}f(n)\ dn
\]
for a suitable class of smooth functions $f$ on $N'(F_S)$. If $S$ consists only of non-archimedean places then
\[
\stint_{N'(F_S)}f(n)\ dn=\int_{N_1'}f(n)\ dn
\]
for any sufficiently large compact open subgroup $N_1'$ of $N'(F_S)$.
(In the archimedean case an ad-hoc definition is given.)

The input for the descent construction is an automorphic representation $\pi$ of $\GL_{2n}(\A)$
which is the isobaric sum $\pi_1\boxplus\dots\boxplus\pi_k$ of pairwise inequivalent irreducible cuspidal representations $\pi_i$ of
$\GL_{2n_i}(\A)$, $i=1,\dots,k$ (with $n_1+\dots+n_k=n$) such that $L^S(\frac12,\pi_i)\ne0$ and $L^S(s,\pi_i,\wedge^2)$ has a pole (necessarily simple)
at $s=1$ for all $i$. Here $L^S(s,\pi_i)$ and $L^S(s,\pi_i,\wedge^2)$ are the standard and exterior square (partial) $L$-functions, respectively.
To any such $\pi$ one constructs a $\psi_{\tilde N}$-generic representation $\tilde\pi$ of $\Mp_n(\A)$, which is called the $\psi_{\tilde N}$-descent of $\pi$.
(See \S\ref{sec: descent} for more details.)

\begin{conjecture}(\cite[Conjecture 1.3]{LMao5}) \label{conj: metplectic global}
Assume that $\tilde\pi$ is the $\psi_{\tilde N}$-descent of $\pi$ as above.
Then for any $\tilde\varphi\in\tilde\pi$ and $\d{\tilde\varphi}\in\d{\tilde\pi}$ and for any sufficiently large finite set $S$ of places of $F$ we have
\begin{multline} \label{eq: globalidentity}
\tilde\whit^{\psi_{\tilde N}}(\tilde\varphi)\tilde\whit^{\psi_{\tilde N}^{-1}}(\d{\tilde\varphi})=2^{-k}(\prod_{i=1}^n\zeta_F^S(2i))\frac{L^S(\frac12,\pi)}{L^S(1,\pi,\sym^2)}\times\\
(\vol(N'(\OO_S)\bs N'(F_S)))^{-1}\stint_{N'(F_S)}(\tilde\pi(n)\tilde\varphi,\d{\tilde\varphi})_{\Sp_n(F)\bs\Sp_n(\A)}\psi_{\tilde N}(n)^{-1}\ dn.
\end{multline}
Here $\zeta_F^S(s)$ is the partial Dedekind zeta function and $\OO_S$ is the ring of $S$-integers of $F$
and $L^S(s,\pi,\sym^2)$ is the symmetric square partial $L$-function of $\pi$.
\end{conjecture}

We recall that the image of the $\psi_{\tilde N}$-descent map consists of the $\psi_{\tilde N}$-generic representations whose
$\psi$-theta lift to $\SO(2n-1)$ vanishes where $\psi$ is determined by $\psi_{\tilde N}$. (See \cite[\S11]{MR2848523} for more details.)
In the case $n=1$, this excludes the so-called exceptional representations.
As was pointed out before, the case $n=1$ goes back to the classical result of Waldspurger on the Fourier coefficients
of half-integral weight modular forms \cite{MR646366}. His result has been revisited extensively --
see \cite{MR894322, MR1244668, MR629468, MR1233447, MR783554, MR1404335, MR2059949, MR2669637, 1308.2353}.
We emphasize that our approach for the case $n=1$ is different from that of Waldpusrger
(who uses the Shimura correspondence) as well as from the relative trace formula approach taken by other authors
(starting with Jacquet \cite{MR983610}). It is likely that the relative trace formula approach carries over to higher rank --
see \cite{MR2656089}. We also mention in this context the recent results of Wei Zhang on the Gan--Gross--Prasad conjecture for unitary groups
and its refinement by Ichino--Ikeda and N. Harris \cite{WeiZhang, MR3164988}.

Let us describe the contents of the paper in more detail.
After setting up some notation in \S\ref{sec: notation} we introduce and study
in \S\ref{sec: metatype} the class $\Irr_{\gen,\meta}\GL_{2n}$ of generic representations of $\GL_{2n}$ over a local field
``of metaplectic type''.
The global analogue of this class is related to a Rankin--Selberg type integral studied by Bump--Friedberg and Friedberg-Jacquet \cite{MR1159108, MR1241129}.
These representations are the input for the descent construction, both locally and globally.

In \S\ref{sec: local conjecture} we introduce, following Ginzburg--Rallis--Soudry, the local Shimura type integrals for
$GL_{2n}\times\Mp_n$ \cite{MR1675971} and the local descent studied in \cite{MR1671452, MR1954940}.

In \S\ref{sec: technical} and \S\ref{sec: local to global} we carry out the reduction to a local statement by piecing together ingredients of the descent constructions.
This is our main result.
More precisely, we show that the relation \eqref{eq: globalidentity} holds up to a constant $\prod_{v\in S}c_{\pi_v}$
which depends only on the local representations $\pi_v$, $v\in S$. (Theorem \ref{thm: local to global}.)
The local constants $c_{\pi_v}$ are certain factors of proportionality
which are explicated in \eqref{eq: main1}.
The main local conjecture (Conjecture \ref{conj: goodeps} which would imply Conjecture \ref{conj: metplectic global}) is that the constants
$c_{\pi_v}$ are equal to the standard epsilon factors $\eps(\frac12,\pi_v,\psi_v)$.

We end the paper by giving a heuristic argument for the local conjecture in the case $n=1$, which is related to the classical Shimura correspondence (\S\ref{sec: n=1}).
This case is already nontrivial and provides many of the ingredients which go into the general case.
It also demonstrates some of the technical difficulties involved in giving a rigorous argument.

The main result of the paper serves as a model for other quasi-split classical groups.
We hope to deal with these cases in the future.

\subsection{Acknowledgement}
The authors would like to thank Moshe Baruch, Joseph Bernstein, Wee Teck Gan, Marcela Hanzer, Atsushi Ichino, Herv\'e Jacquet, Dihua Jiang, Omer Offen,
Ravi Raghunathan, Yiannis Sakellaridis, Freydoon Shahidi, David Soudry, Marko Tadi\'c and Akshay Venkatesh for helpful discussions and suggestions.
We thank Joachim Schwermer and the Erwin Schr\"odinger Institute in Vienna for their support and for providing excellent working conditions
for collaboration.

\section{Notation and preliminaries} \label{sec: notation}
For the convenience of the reader we introduce in this section the most common notation
that will be used throughout.

We fix a positive integer $n$.
(In the body of the paper, we sometimes use $n$ also as a running variable.
We hope that this does not cause any confusion.)
The letters $i$, $j$, $k$, $l$, $m$ will denote auxiliary positive integers.

Let $F$ be a local field of characteristic $0$.

\subsection{Groups, homomorphisms and group elements} \label{sec: elements}

All algebraic groups are defined over $F$.
By abuse of notation we often write $X=X(F)$ for any variety over $F$.

\begin{itemize}

\item $I_m$ is the identity matrix in $\GL_m$, $w_m$ is the $m\times m$-matrix with ones on
the non-principal diagonal and zeros elsewhere; $J_m:=\sm{}{w_m}{-w_m}{}\in\GL_{2m}$.

\item For any group $Q$, $Z_Q$ is the center of $Q$; $e$ is the identity element of $Q$.
We denote the modulus function of $Q$ by $\modulus_Q$.

\item $\GLnn=\GL_{2n}$, $\GLn=\GL_n$.

\item $G=\Sp_{2n}=\{g\in\GL_{4n}:\, g^tJ_{2n}g=J_{2n}\}$ where $g^t$ is the transpose of $g$.

\item $G'=\Sp_n=\{g\in\GL_{2n}:\,g^tJ_ng=J_n\}$.

\item $\tilde G=\Mp_n$ is the metaplectic group, i.e., the two-fold cover of $G'$.
We write elements of $\tilde G$ as pairs $(g,\epsilon)$, $g\in G$, $\epsilon=\pm1$
where the multiplication is given by Rao's cocycle. (Cf. \cite{MR2848523}.)

\item $K$ is the standard maximal compact subgroup of $G$; similarly for $K'$.
When $F$ is $p$-adic, and $p\ne2$ we view $K'$ as a subgroup of $\tilde G$.

\item $P=\Levi\ltimes U$ is the Siegel parabolic subgroup of $G$, with its standard Levi decomposition.

\item $\bar P=P^t$ is the opposite parabolic of $P$, with unipotent radical $\bar U=U^t$.

\item $N$ (resp., $N_\GLnn$) is the standard maximal unipotent subgroup of $G$ (resp., $\GLnn$) consisting of upper unitriangular matrices;
$T$ (resp., $T_\GLnn$) is the maximal torus of $G$ (resp., $\GLnn$) consisting of diagonal matrices;
$B=T\ltimes N$ is the Borel subgroup of $G$.

\item $G'$ is embedded as a subgroup of $G$ via $g\mapsto\toG(g)=\diag(I_n,g,I_n)$.

\item The Siegel parabolic subgroup $P'$ of $G'$ satisfies $\toG(P')=P\cap \toG(G')$.
Similarly define $\Levi'$, $U'$, $N'$, $T'$ and $B'$
so that $P'=\Levi'\ltimes U'$ and $B'=T'\ltimes N'$.

\item For a subgroup $X$ of $G$, $X_\Levi$ denotes $X\cap\Levi$. In particular $T=T_\Levi$
and $N_\Levi=\levi(N_\GLnn)$. Similarly $X'_{\Levi'}=X'\cap \Levi'$ for a subgroup $X'$ of $G'$.

\item $g\mapsto g^*$ is the outer automorphism of $\GL_m$ given by $g^*=w_m^{-1}\,(g^t)^{-1} w_m$.

\item We use the isomorphism $\levi(g)=\diag(g,g^*)$ to identify $\GLnn$ with $\Levi$.
Similarly for $\levi':\GLn\rightarrow\Levi'$.

\item For any subgroup $X\subset G$, $X_\GLnn$ denotes the subgroup $\levi^{-1}(X_\Levi)$ of $\GLnn$. Similarly,
for $X'_{\GLn}$ if $X'\subset G'$.

\item We use the embeddings $\toM(g)=\diag(g,I_n)$ and $\toMd(g)=\diag(I_n,g)$ to identify $\GLn$ with subgroups of $\GLnn$.
We also set $\toLevi=\levi\circ\toM$ and $\toLevid=\levi\circ\toMd=\toG\circ\levi'$.

\item When $g\in G'$, we write $\animg{g}=(g,1)\in\tilde G$. (Of course, $g\mapsto\animg{g}$ is not a group homomorphism.)

\item $\tilde N$ is the inverse image of $N'$ under the canonical projection $\tilde G\rightarrow G'$.
We will identify $N'$ with a subgroup of $\tilde N$ via $n\mapsto\animg{n}$.

\item $\xi_m=(0,\ldots,0,1)\in F^m$.

\item $E=\operatorname{diag}[(-1)^{i-1}]\in \GLnn$. $H$ is the centralizer of $\levi(E)$ in $G$. It is isomorphic to $\Sp_n\times\Sp_n$.
$H_\GLnn$ is then the centralizer of $E$ in $\GLnn$. It is isomorphic to $\GLn\times\GLn$.

\item $\Mat_{l,m}$ is the vector space of $l\times m$ matrices over $F$.

\item $x\mapsto \startran{x}$ is the twisted transpose map on $\Mat_{m,m}$ given by
$\startran{x}=w_m x^t w_m$.

\item $\symspace_m=\{x\in\Mat_{m,m}:\startran{x}=x\}$.

\item $\toU:\symspace_{2n}\rightarrow U$ is the isomorphism given by $\toU(x)=\sm{I_{2n}}{x}{}{I_{2n}}$.
Similarly $\toUbar(x)=\sm{I_{2n}}{}x{I_{2n}}$ is the isomorphism from $\symspace_{2n}$ to $\bar U$.

\item $\toU_\GLnn:\Mat_{n,n}\rightarrow N_\GLnn$ is the homomorphism given by $\toU_\GLnn(x)=\sm{I_{n}}{x}{}{I_{n}}$.

\item $\one_{i,j}\in \Mat_{l,m}$ is the matrix with one at the $(i,j)$-entry and zeros elsewhere.

\item $\one^{\symspace_m}_{i,j}=\one_{i,j}+\startran{\one_{i,j}}\in \symspace_m$ when $i+j\not=m+1$ and
$\one^{\symspace_m}_{i,m+1-i}=\one_{i,m+1-i}$.

\item $w_U=\sm{}{I_{2n}}{-I_{2n}}{}\in G$ represents the longest $\Levi$-reduced Weyl element of $G$.

\item $w'_{U'}=\sm{}{I_n}{-I_n}{}\in G'$ represents the longest $\Levi'$-reduced Weyl element of $G'$.

\item $\wnn=w_{2n}\in \GLnn$ represents the longest Weyl element of $\GLnn$.

\item $\gamma=w_U\toG(w'_{U'})^{-1}=\left(\begin{smallmatrix}&I_n&&\\&&&I_n\\-I_n&&&\\&&I_n&\end{smallmatrix}\right)\in G$.

\item $V$ (resp., $V_0$) is the unipotent radical in $G$ of the standard parabolic subgroup with Levi
$\GL_1^n\times\Sp_n$ (resp., $\GL_1^{n-1}\times\Sp_{n+1}$).
Thus $N=V\rtimes\toG(N')$, $V_0$ is normal in $V$ and $V/V_0$ is isomorphic to the Heisenberg group
of dimension $2n+1$.

\item $V_\gamma=V\cap\gamma^{-1}N\gamma=\toG(w'_{U'})(N_\Levi\cap V)\toG(w'_{U'})^{-1}=\toLevi(N'_\GLn)\ltimes\{\toU(\sm{x}{}{}{\startran{x}}):x\in\Mat_{n,n}\}$.

\end{itemize}

\subsection{Characters} \label{sec: characters}
We fix a non-degenerate character $\psi_{N_\GLnn}$ of $N_\GLnn$.
This character will determine characters of several other unipotent group, as well as of $F$, as follows:
\begin{itemize}
\item $\psiweil$ is the non-trivial character of $F$ given by $\psiweil(x)=\psi_{N_\GLnn}(I_{2n}+x\one_{n,n+1})$.

\item $\psi_{N_\Levi}$ is the non-degenerate character of $N_\Levi$ such that $\psi_{N_\Levi}(\levi(u))=\psi_{N_\GLnn}(u)$.

\item $\psi_{N'_{\Levi'}}$ is the non-degenerate character of $N'_{\Levi'}$ given by
$\psi_{N'_{\Levi'}}(\levi'(u))=\psi_{N_\GLnn}(\toM(u'))$ for $u'\in N'_\GLn$.
Thus, $\psi_{N'_{\Levi'}}(n)=\psi_{N_\Levi}(\gamma \toG(n)\gamma^{-1})$.

\item $\psi_{U'}$ is the character on $U'$ given by $\psi_{U'}(u)=\psiweil(\frac12u_{n,n+1})^{-1}$.

\item $\psi_{\tilde N}$ is the genuine character of $\tilde N$ whose restriction to $N'$ is the non-degenerate character
\[
\psi_{\tilde N}(nu)=\psi_{N'_{\Levi'}}(n)\psi_{U'}(u),\,\,n\in N'_{\Levi'},\,u\in U'.
\]

\item $\psi_N$ is the \emph{degenerate} character on $N$ given by
$\psi_N(nu)=\psi_{N_\Levi}(n)$ for any $n\in N_\Levi$ and $u\in U$.

\item $\psi_V$ is the character of $V$, trivial on $V\cap U$ and $\levi(\toU_\GLnn(\Mat_{n,n}))$, given by
$\psi_V(\toLevi(n))=\psi_{N_\Levi}(\toLevid(n^*))$ for $n\in N'_{\GLn}$.

\end{itemize}

For convenience, we will fix a non-trivial character $\psi$ of $F$ and set $\psi_a=\psi(a\cdot)$. Let
\[
\psi_{N_\GLnn}(u)=\psi(u_{1,2}+\dots+u_{2n-1,2n}).
\]
Thus $\psiweil=\psi$. This choice is different from the conventions of \cite{MR2848523} (which is in turn different from
the conventions of \cite{MR1722953, MR1671452}).
This is a minor issue which will be addressed in \S\ref{sec: local to global}. (See Remark~\ref{rem: diffgrs}).
With this choice of $\psi_{N_\GLnn}$ we have
\begin{eqnarray*}
\psi_{N'_{\Levi'}}(\levi'(u'))&=&\psi(u'_{1,2}+\dots+u'_{n-1,n}), \ \ u'\in N'_{\GLn}\\
\psi_V(v)&=&\psi(v_{1,2}+\dots+v_{n-1,n})^{-1}, \ \ v\in V\cap\Levi.
\end{eqnarray*}

\subsection{Measures}\label{sec: chevalley}

We take the self-dual Haar measure on $F$ with respect to $\psi$.
We use the following convention for Haar measures for algebraic subgroups of $G$, (we consider $G'$, $\GLnn$, $\GLn$ as
subgroups of $G$ through the embeddings $\toG$, $\levi$ and $\toG\circ\levi'$).

When $F$ is a $p$-adic field, let $\OO$ be its ring of integers.
The Lie algebra $\mathfrak{M}$ of $\GL_{4n}$ consists of the $4n\times 4n$-matrices $X$ over $F$.
Let $\mathfrak{M}_{\OO}$ be the lattice of integral matrices in $\mathfrak{M}$.
For any algebraic subgroup $Q$ of $\GL_{4n}$ defined over $F$
(e.g., an algebraic subgroup of $G$)
let $\mathfrak{q}\subset\mathfrak{M}$ be the Lie algebra of $Q$.
The lattice $\mathfrak{q}\cap \mathfrak{M}_{\OO}$ of $\mathfrak{q}$ gives rise to a gauge form of $Q$
(determined up to multiplication by an element of $\OO^*$) and we use it to define a Haar measure on $Q$ by the recipe of \cite{MR0217077}.

When $F=\R$, the measures are fixed similarly, except that we use $\mathfrak{M}_{\Z}$ in place of $\mathfrak{M}_{\OO}$.
When $F=\C$ we view the group as a group over $\R$ by restriction of scalars and apply the above convention.

\subsection{Weil representation}
Let $W$ be a symplectic space over $F$ with symplectic form $\sprod{\cdot}{\cdot}$.
Let $\Heise=\Heise_W$ be the Heisenberg group of $(W,\sprod{\cdot}{\cdot})$.
Recall that $\Heise_W=W\oplus F$ with the product rule
\[
(x,t)\cdot (y,z)=(x+y,t+z+\frac12\sprod xy).
\]
Fix a polarization $W=W_+\dsum W_-$. The group $\Sp(W)$ acts on the right on $W$.
We write a typical element of $\Sp(W)$ as
$\sm ABCD$ where $A\in \Hom(W_+,W_+)$, $B\in \Hom(W_+,W_-)$, $C\in \Hom(W_-,W_+)$ and $D\in \Hom (W_-,W_-)$.
Let $\Mp(W)$ be the metaplectic two-fold cover of $\Sp(W)$ with respect to the Rao cocycle determined
by the splitting.
Consider the Weil representation $\we$ of the group $\Heise_W\rtimes\Mp(W)$ on $\swrz(W_+)$.
Explicitly, for any $\Phi\in\swrz(W_+)$ and $X\in W_+$ we have
\begin{subequations}
\begin{align}
\label{eq: weilH1} \we(a,0)\Phi(X)&=\Phi(X+a), \ \ a\in W_+,\\
\label{eq: weilH2} \we(b,0)\Phi(X)&=\psi(\sprod Xb)\Phi(X),\ \ b\in W_-,\\
\label{eq: weilH3} \we(0,t)\Phi(X)&=\psi(t)\Phi(X), \ \ t\in F,\\
\we(\sm g00{g^*},\epsilon)\Phi(X)&=\epsilon\weilfctr(\det g)\abs{\det g}^{\frac12}\Phi(X g), \ \ g\in\GL(W_+),\\
\we(\sm IB0I,\epsilon)\Phi(X)&=\epsilon\psi(\frac12\sprod{X}{XB})\Phi(X),\ \ B\in\Hom(W_+,W_-)\text{ self-dual},
\end{align}
\end{subequations}
where $\weilfctr$ is Weil's factor.\footnote{Strictly speaking, $\weilfctr(a)=\gamma(\psi_a)/\gamma(\psi)$ where
$\gamma(\psi)$ is Weil's index, which is an eighth root of unity.
We caution that the $\gamma_\psi$ in \cite[(1.4)]{MR2848523} is really $\gamma_{\psi_{\frac12}}$ in our notation.}

We now take $W=F^{2n}$ with the standard symplectic form
\[
\sprod{(x_1,\dots,x_{2n})}{(y_1,\dots,y_{2n})}=\sum_{i=1}^nx_iy_{2n+1-i}-\sum_{i=1}^ny_ix_{2n+1-i}
\]
and the standard polarization $W_+=\{(x_1,\dots,x_n,0,\dots,0)\}$, $W_-=\{(0,\dots,0,y_1,\dots,y_n)\}$.
(We identify $W_+$ and $W_-$ with $F^n$.)
The corresponding Heisenberg group is isomorphic to the quotient $V/V_0$ via $v\mapsto\Hei{v}:=((v_{n,n+j})_{j=1,\dots,2n},\frac12v_{n,3n+1})$;
(recall that $V$ and $V_0$ are the unipotent groups defined in \S\ref{sec: elements}.)

For $X=(x_1,\dots,x_n),X'=(x_1',\dots,x_n')\in F^n$ define
\[
\sprod X{X'}'=x_1x'_n+\dots+x_nx'_1.
\]
For $\Phi\in\swrz(F^n)$ define
\[
\hat\Phi(X)=\int_{F^n}\Phi(X')\psi(\sprod{X}{X'}')\ dX'.
\]

Then, the Weil representation is realized on $\swrz(F^n)$ as follows.
\begin{subequations}
\begin{align}
\label{eq: weil1} \we(\animg{\levi'(h)})\Phi(X)&=\abs{\det(h)}^{\frac{1}{2}}
\beta_{\psi}(\levi'(h))\Phi(Xh),\ \ h\in\GLn,\\
\label{eq: weil2} \we(\animg{w'_{U'}})\Phi(X)&=\beta_{\psi}(w'_{U'})\hat\Phi(X),\\
\label{eq: weil3} \we(\animg{\sm1B{}1})\Phi(X)&=\psi(\frac12\sprod{X}{XB}')\Phi(X),\ \ \ B\in\symspace_n.
\end{align}
\end{subequations}
Here $\beta_\psi(g)$, $g\in G'$ are certain roots of unity.

Starting with $\psi_{N_\GLnn}$, and with $\psiweil$ and $\psi_V$ as above,
we extend $\weil_{\psiweil}$ to a representation $\wev$ of $V\rtimes\tilde G$ by setting
\begin{equation}\label{eq: weilext}
\wev(v \animg g)\Phi=\psi_V(v)\weil_{\psiweil}(\Hei{v})(\weil_{\psiweil}(\animg g)\Phi),\,\,v\in V,\,\,g\in G'.
\end{equation}
Then for any $g\in G'$, $v\in V$ we have
\begin{equation}\label{eq: weilext2}
\wev((\toG(g)v\toG(g)^{-1})\animg g)\Phi=\wev(\animg g)(\wev(v)\Phi).
\end{equation}

\subsection{Other local notation}
\begin{itemize}
\item Let $\Delta_{\Sp_n}(s)=\prod_{i=1}^n\zeta_F(s+2i-1)$ where $\zeta_F(s)$ is the local Tate factor.

\item We denote by $\Irr Q$ (resp., $\Irr_{\cusp}Q$, $\Irr_{\sqr}Q$, $\Irr_{\temp}Q$) the class of irreducible
(and resp., supercuspidal, square-integrable, tempered) representations of $Q$.
If $Q$ is quasi-split we write $\Irr_{\gen}Q$ for the class of generic representations
(or $\Irr_{\gen,\chi}Q$ if we want to keep track of the generic character $\chi$).

\item If $\pi$ is an irreducible generic representation of $\GLnn$ we use $\WhitM(\pi)$ to denote the (uniquely
determined) Whittaker space of $\pi$ with respect to the character $\psi_{N_\GLnn}$. Similarly we use the notation
$\WhitMd$, $\WhitML$, $\WhitMLd$, $\WhitG$, $\WhitGd$.

\item If $\pi$ is an irreducible generic representation of $\Levi$ let $\Ind(\WhitML(\pi))$ be the space of smooth left $U$-invariant functions $W:G\rightarrow\C$ such that
for all $g\in G$, the function $\modulus_P(m)^{-\frac12}W(mg)$ on $\Levi$ belongs to $\WhitML(\pi)$. Similarly define $\Ind(\WhitMLd(\pi))$,
$\Ind(\WhitM(\pi))$ and $\Ind(\WhitMd(\pi))$.

\end{itemize}

\subsection{Global case}
Suppose now that $F$ is a number field and $\A$ is its ring of adeles. Most of the previous notation has an obvious meaning in the global
context. Some changes in the global context are as follows.
\begin{itemize}
\item We write $\Cusp Q$ for the set of equivalence classes of irreducible cuspidal representations of $Q(\A)$.

\item $\K=\prod K_v$ where $K_v$ are as in the local case.

\item $\tilde G(\A)$ is the metaplectic two-fold cover of $G'(\A)$ which splits over $G'(F)$.

\item $\Delta_{\Sp_n}^S(s)=\prod_{i=1}^n\zeta_F^S(s+2i-1)$ for any finite set of places.

\item $\psi$ is a fixed non-trivial character of $F\bs\A$.
The characters in \S\ref{sec: characters} are defined on the adelic points of the groups (and they are trivial
on the $F$-points).

\item We will take Tamagawa measures on the adelic groups.

\item For $\Phi\in\swrz(\A^n)$ define the theta function
\[
\Theta^\Phi_{\psi_{N_\GLnn}}(v \animg g)=\sum_{\xi\in F^n}\wev(v \animg g)\Phi(\xi),\ \ v\in V(\A), g\in G'(\A).
\]

\item For any automorphic form $\varphi$ on $G(\A)$ and $\Phi\in\swrz(\A^n)$, let
$\FJ(\varphi,\Phi)$ be the Fourier--Jacobi coefficient (a genuine function on $\tilde G(\A)$)
\begin{equation} \label{def: FJ}
\FJ(\varphi,\Phi)(\animg g)=\int_{V(F)\bs V(\A)}\varphi(v\toG(g))\Theta^\Phi_{\psi_{N_\GLnn}^{-1}}(v\animg g)\,dv,\ \ g\in G'(\A).
\end{equation}
\item For an automorphic form $\tilde\varphi$ on $\tilde G(\A)$, let
\[
\tilde\whit^{\psi_{\tilde N}}({\tilde\varphi},\tilde g)=\tilde\whit^{\psi_{\tilde N}}_{\tilde\varphi}(\tilde g)=
\int_{N'(F)\bs  N'(\A)}\tilde\varphi(\animg n\tilde g)\psi_{\tilde N}( n)^{-1}\ d n, \ \ g\in G'(\A).
\]
We also write
$\tilde\whit^{\psi_{\tilde N}}(\tilde\varphi)=\tilde\whit^{\psi_{\tilde N}}_{\tilde\varphi}(e)$.
Similarly define $\whit^{\psi_N}({\varphi},g)=\whit^{\psi_N}_{\varphi}(g)$, $\whit^{\psi_N}(\varphi)$ etc.
\end{itemize}

\section{Representations of metaplectic type} \label{sec: metatype}
In this section we consider a certain class of irreducible representations of the general linear group of even rank,
both locally and globally.
The descent construction is applied to this class of representations.
In fact, roughly speaking these representations are expected to be the functorial image of $\Mp_n$ to $\GL_{2n}$.

Recall that $\GLnn=\GL_{2n}$ and $H_\GLnn$ is the centralizer of $E=\diag(1,-1,\dots,1,-1)$,
isomorphic to $\GL_n\times\GL_n$. The involution $\wnn$ lies in the normalizer of $H_\GLnn$.
(In fact, since the $\GLnn$-conjugacy class of $E$ intersects the center of $H_\GLnn$ at $\{\pm E\}$,
$H_\GLnn$ is of index two in its normalizer.)

\subsection{Local setting}
Let $F$ be a local field.
\begin{definition}
We say that $\pi\in\Irr\GLnn$ is of \emph{metaplectic type} if there exists a continuous non-zero
$H_\GLnn$-invariant linear form $\ell$ on the space of $\pi$.
We denote the set of $\pi\in\Irr\GLnn$ of metaplectic type by $\Irr_{\meta}\GLnn$.
\end{definition}

Clearly, any $\pi\in\Irr_{\meta}\GLnn$ has a trivial central character, i.e., it factors through $\PGL_{2n}$.
Moreover, it is known that if $\pi\in\Irr_{\meta}\GLnn$ then $\pi$ is self-dual
and $\ell$ is uniquely determined up to a scalar (\cite{MR1394521} -- $p$-adic case; \cite{MR2553879} -- archimedean case).
Since the involution $\wnn\in \GLnn$ normalizes $H_\GLnn$ we have
$\ell\circ\pi(\wnn)=\eps_\pi\ell$ where $\eps_\pi\in\{\pm1\}$ does not depend on the choice of $\ell$.

It will be convenient to set
\[
\Irr_{\meta,\eps}\GLnn=\{\pi\in\Irr_{\meta}\GLnn:\eps_\pi=\epsilon(\frac12,\pi,\psi)\}
\]
where $\epsilon(s,\pi,\psi)$ is the standard $\epsilon$-factor attached to $\pi$.
Note that this does not depend on the choice of $\psi$ since any $\pi\in\Irr_{\meta}\GLnn$ has a trivial central character.
We expect that in fact $\Irr_{\meta,\eps}\GLnn=\Irr_{\meta}\GLnn$.
Below we will prove a special case of this in the $p$-adic case. More precisely, we will show below
\begin{theorem} \label{thm: genmetaeps}
Suppose that $F$ is $p$-adic. Then
\begin{equation} \label{eq:shalmetagen}
\Irr_{\gen,\meta}\GLnn=\Irr_{\gen,\meta,\eps}\GLnn.
\end{equation}
\end{theorem}

\subsection{The case $n=1$}\label{S:metan=1}
For completeness we study the case of $\PGL_2$.
Let $\pi\in\Irr\PGL_2$ (viewed as an irreducible representation of $\GL_2$ with trivial central character).
If $\pi$ is generic, i.e., infinite dimensional,
then $\pi\in \Irr_{\meta}\GL_2$, i.e., it admits a non-trivial continuous linear form invariant under the diagonal torus.
On the Whittaker model this functional is given by
\[
\ell(W)=\frac{\int_{F^*} W(\sm t001)\abs{t}^s\ dt}{L(s+\frac12,\pi)}\big|_{s=0}, \ \ \ W\in\WhitM(\pi)
\]
in the sense of analytic continuation. In fact, $L(s,\pi)$ is holomorphic at $s=\frac12$ in this case.
The functional equation (in the sense of analytic continuation)
\[
\int_{F^*}W(\sm t001\sm {}11{})\abs{t}^{-s}\ dt=
\gamma(s+\frac12,\pi,\psi)\int_{F^*} W(\sm t001)\abs{t}^s\ dt,\ \ \ W\in\WhitM(\pi)
\]
shows that $\eps_\pi=\epsilon(\frac12,\pi,\psi)$.

\begin{lemma}
Suppose that $\pi\in\Irr\GL_2$ is non-generic (i.e., finite dimensional) with trivial central character.
\begin{enumerate}
\item In the $p$-adic case $\pi=\chi\circ\det$ where $\chi$ is a quadratic character of $F^*$.
Moreover, $\pi\in\Irr_{\meta}\GL_2$ if and only if $\chi$ is trivial, in which case $\eps_\pi=\epsilon(\frac12,\pi,\psi)=1$.
\item Suppose that $F=\R$ and let $p=\dim\pi$. Then $p$ is odd. Let $\pi_p$ be the restriction
to $\GL_2(\R)$ of the unique irreducible $p$-dimensional algebraic representation of $\GL_2(\C)$ with trivial central character.
Then either $\pi=\pi_p$ or $\pi=\pi_p\otimes\sgn$ where $\sgn$ is the signum character. In the first case $\pi\in\Irr_{\meta}\GL_2$ and
$\eps_\pi=(-1)^{(p-1)/2}=\epsilon(\frac12,\pi,\psi)$. In the second case $\pi\notin\Irr_{\meta}\GL_2$.
\item Suppose that $F=\C$. Then there exist unique positive integers $m$ and $k$ such that $\pi\rest_{\SL_2(\C)}=\sigma_m\otimes\overline{\sigma_k}$
where $\sigma_m$ is the algebraic $m$-dimensional irreducible representation of $\SL_2(\C)$ and $k+m$ is even.
Moreover, $\pi\in\Irr_{\meta}\GL_2$ if and only if both $m$ and $k$ are odd and in this case
$\eps_{\pi}=(-1)^{(m-k)/2}=\epsilon(\frac12,\pi,\psi)$.
\end{enumerate}
Thus, in all cases $\Irr_{\meta}\GL_2=\Irr_{\meta,\eps}\GL_2$.
\end{lemma}

\begin{proof}
The $p$-adic case is well known.

Note that if $\pi$ is the Langlands quotient of $\Ind(\mu,\mu^{-1})$ then $\epsilon(\frac12,\pi,\psi)=\mu(-1)$.

For the case $F=\R$, every finite-dimensional $\pi\in\Irr\GL_2$ is of the form $\sigma\otimes\chi$
where $\sigma$ is an algebraic representation of $\GL_2$ and $\chi$ is a character of $\R^*$.
Let $\omega_\sigma$ be the central character of $\sigma$.
If $\pi$ has a trivial central character then $\omega_\sigma(-1)=1$ and therefore $\sigma$ is odd dimensional.
By twisting $\sigma$ by a suitable power of $\det$ we can also assume that $\omega_\sigma=1$ and therefore
$\chi^2=1$, i.e. $\chi=1$ or $\chi=\sgn$.
Write $p=2m+1$ and note that $\pi_p$ is the Langlands quotient of
$\Ind(\abs{\cdot}^{m+\frac12},\abs{\cdot}^{-(m+\frac12)})\otimes\sgn^m$.
Therefore, by the above $\epsilon(\frac12,\pi_p,\psi)=(-1)^m$.
On the other hand, it is clear that the algebraic representation $\pi_p$ has a torus invariant vector and that $\eps_{\pi_p}=(-1)^m$.

In the complex case, the finite-dimensional representations of $\SL_2(\C)$ are of the form
$\sigma_m\otimes\overline{\sigma_k}$ for unique positive integers $m$ and $k$.
This can be viewed as the Langlands quotient of the parabolic induction of the character
$\sm z{}{}{z^{-1}}\mapsto z^m\bar z^k$.
It descends to $\operatorname{PSL}_2(\C)=\PGL_2(\C)$ if and only if $m$ and $k$ have the same parity.

Thus we can write $\pi$ as the Langlands quotient of $\Ind(\mu,\mu^{-1})$ where $\mu^2(z)=z^m\bar z^k$.
Therefore, $\epsilon(\frac12,\pi,\psi)=\mu(-1)=(-1)^{(m-k)/2}$.

It is clear that $\pi\in\Irr_{\meta}\GL_2$ if and only if both $\sigma_m$ and $\sigma_k$
admit a vector invariant under the torus, i.e., if and only if both $m$ and $k$ are odd,
and in this case $\eps_\pi=(-1)^{(m-1)/2}(-1)^{(k-1)/2}$.
\end{proof}

\subsection{Shalika functionals}
A closely related notion to metaplectic type is a Shalika functional.
Let $\Shal$ be the Shalika subgroup
\[
\Shal=\{\sm g{}{}g\sm{I_n}X{}{I_n}:g\in\GLn, X\in\Mat_{n,n}\}
\]
of $\GLnn$ with the character $\psi_{\Shal}(\sm g{}{}g\sm{I_n}X{}{I_n})=\psi(\tr X)$.
By definition, a Shalika functional is a non-trivial continuous functional $ L$ on $\pi$ such that
$ L(\pi(s)v)=\psi_{\Shal}(s)L(v)$ for all $s\in\Shal$ and $v\in V_\pi$.
We say that $\pi\in\Irr\GLnn$ is of \emph{Shalika type} if it admits a Shalika functional.
(This property is independent of the choice of $\psi$.)
Once again, a Shalika functional is unique up to a scalar if it exists (\cite{MR1394521, MR2552255}).
We denote the class of irreducible representations of Shalika type by $\Irr_{\shal}\GLnn$.
It is known that $\Irr_{\shal}\GLnn\subset\Irr_{\meta,\eps}\GLnn$ (cf.~\cite[\S3]{MR1241129}
and in particular the functional equation of [loc. cit., Proposition 3.3] for the $p$-adic case and
\cite{MR2552255} for the archimedean case).
Also note that for $n=1$ we have $\Irr_{\shal}\GLnn=\Irr_{\gen}\PGL_2$.
Of course, $\Irr_{\shal}\GLnn\ne\Irr_{\meta}\GLnn$, e.g.~the trivial representation
belongs to $\Irr_{\meta}\GLnn-\Irr_{\shal}\GLnn$.

\subsection{}
Let $I(\pi,s)$ be the representation of $G$ parabolically induced from $P$ and $\pi\otimes\abs{\det\cdot}^s$,
viewed as a representation of $\Levi$ through $\levi$.
We recall that if $\pi\in\Irr_{\meta}\GLnn$ then $I(\pi,\frac12)$ admits a non-trivial $H$-invariant functional; here
$H$ is the centralizer of $\levi(E)$ in $G$.
This is proved in the course of the proof of \cite[\S3.3, Theorem 2]{MR1671452}. More precisely, if $\ell$ is an
$H_{\GLnn}$-invariant functional on $\pi$ then
\begin{equation} \label{eq: Hinv}
\varphi\rightarrow\int_{H\cap P\bs H}\ell(\varphi(h))\ dh\text{ defines an $H$-invariant functional on }I(\pi,\frac12).
\end{equation}
This is well defined since $\modulus_{H\cap P}(\levi(m))=\modulus_P^{\frac12}(\levi(m))\abs{\det m}^{\frac12}$ for all $m\in H_\GLnn$.
\label{sec: Hdist}

We recall the local $L$-factor $L(s,\pi,\wedge^2)$ defined by Shahidi -- cf.~\cite{MR1159430}.

We record the following assertion which follows from results in the literature.
\begin{proposition} \label{prop: metasqr}
Suppose $\pi\in\Irr_{\sqr}\GLnn$ and $F$ is $p$-adic. Then the following conditions are equivalent.
\begin{enumerate}
\item $L(0,\pi,\wedge^2)=\infty$.
\item $I(\pi,\frac12)$ is reducible.
\item $\pi\in\Irr_{\shal}\GLnn$.
\item $\pi\in\Irr_{\meta}\GLnn$.
\end{enumerate}
In particular, $\Irr_{\sqr,\meta}\GLnn\subset\Irr_{\meta,\eps}\GLnn$.
\end{proposition}

\begin{proof}
The equivalence of the first two conditions (in fact for any $\pi\in\Irr_{\temp}\GLnn$, also in the archimedean case)
follows from \cite[Theorem 4.1 and Proposition 5.3]{MR1634020} and \cite{MR1637097}.

The implication 1$\implies$3 (for $\pi\in\Irr_{\sqr}\GLnn$ and $F$ $p$-adic) follows from \cite[Corollary 4.4]{MR2846403}
and the main result of \cite{MR3008415}.

If $\pi\in\Irr_{\meta,\sqr}\GLnn$ (or in fact, if $\pi\in\Irr_{\meta,\gen}\GLnn$) then the reducibility of $I(\pi,\frac12)$ follows from
\cite[\S3.3, Theorem 1]{MR1671452} and the fact that $I(\pi,\frac12)$ admits a non-trivial $H$-invariant linear form.

Since any $\pi\in\Irr_{\shal}\GLnn$ belongs to $\Irr_{\meta}\GLnn$ the proposition follows.
\end{proof}

Denote by $\times$ parabolic induction for the general linear group.

\begin{lemma}[See also \cite{MR1954940}] \label{lem: simpleind}
Suppose that $n=n_1+\dots+n_k$, $\pi_1,\dots,\pi_k\in\Irr_{\meta}\GL_{2n_i}$, $i=1,\dots,k$  and
$\pi=\pi_1\times\dots\times\pi_k$ is irreducible.
Then $\pi\in\Irr_{\meta}\GLnn$ and $\eps_\pi=\eps_{\pi_1}\dots\eps_{\pi_k}$.
In particular, if $\pi_1,\dots,\pi_k\in\Irr_{\meta,\eps}\GL_{2n_i}$ then $\pi\in\Irr_{\meta,\eps}\GLnn$.

Similarly, let $\sigma\in\Irr\GLn$ and suppose that $\pi=\sigma\times\d\sigma\in\Irr\GLnn$.
Then $\pi\in\Irr_{\meta,\eps}\GLnn$, so that $\eps_\pi=\omega_\sigma(-1)$
where $\omega_\sigma$ is the central character of $\sigma$.
\end{lemma}

\begin{proof}
For the first part we recall the argument of \cite[\S5.5]{MR1954940}.
Let $\new P=\new M\ltimes\new U$ be the parabolic subgroup of type $(2n_1,\dots,2n_k)$ of $\GLnn$
so that $\pi$ is the parabolic induction of $\tau=\pi_1\otimes\dots\otimes\pi_k$ from $\new P$.
Let $w_0^{\new M}$ be the longest Weyl element in $\new M$.
Let $\beta_i$ be non-trivial functionals on $\pi_i$ which are invariant under
$\GL_{n_i}\times\GL_{n_i}$, $i=1,\dots,k$.
Let $\beta$ be the functional $\beta_1\otimes\dots\otimes\beta_k$ on $\tau$.
We have $\beta\circ\tau(w_0^{\new M})=\eps_\tau\beta$ where $\eps_\tau=\eps_{\pi_1}\dots\eps_{\pi_k}$.
Note that $\new P\cap H_\GLnn$ is a parabolic subgroup of $H_\GLnn$ and $\modulus_{\new P\cap H_\GLnn}=\modulus_{\new P}^{\frac12}\rest_{\new P\cap H_\GLnn}$.
Therefore, $\lambda(\varphi):=\int_{\new P\cap H_\GLnn\bs H_\GLnn}\beta(\varphi(h))\ dh$
is a non-zero $H_\GLnn$-invariant form on $\pi$. Since $\wnn\in H_\GLnn w_0^{\new M}$, we have
\begin{multline*}
\lambda(\pi(\wnn)\varphi)=\lambda(\pi(w_0^{\new M})\varphi)=
\int_{\new P\cap H_\GLnn\bs H_\GLnn}\beta(\varphi(hw_0^{\new M}))\ dh=
\int_{\new P\cap H_\GLnn\bs H_\GLnn}\beta(\varphi(w_0^{\new M} h))\ dh\\=\int_{\new P\cap H_\GLnn\bs H_\GLnn}\beta(\tau(w_0^{\new M})\varphi(h))\ dh=
\eps_{\tau}\lambda(\varphi).
\end{multline*}

For the second part, let $\beta$ be the canonical pairing on $\sigma\otimes\d\sigma$.
Consider the parabolic subgroup $\new P=\new M\ltimes\new U$ of type $(n,n)$.
We construct a family $\lambda(s)$ of $\new M$-invariant functionals on $\sigma\otimes\abs{\det}^s\times
\d\sigma\otimes\abs{\det}^{-s}$ by setting
\[
\lambda(\varphi,s):=\int_{\new M_\nu\bs\new M}\beta(\varphi_s(\nu h))\ dh
\]
where $w:=\sm{}{I_n}{I_n}{}$, $\nu:=w\sm{I_n}{I_n}{}{I_n}=\sm{}{I_n}{I_n}{I_n}$
and $\new M_\nu=\new M\cap\nu^{-1}\new P\nu=\new M\cap\nu^{-1}\new M\nu=\{\sm g{}{}g:g\in\GLn\}$.
By general results (in the context of symmetric spaces) of Blanc--Delorme \cite{MR2401221} in the $p$-adic case
and Brylinski--Delorme \cite{MR1176208} in the archimedean case,
$\lambda(s)$ converges for $\Re s\gg0$ and admits a meromorphic continuation.\footnote{At least the convergence part
can be easily checked directly in this case.}

Since $w$ normalizes both $\new M$ and $\new M_\nu$ and $\nu w\nu^{-1}=\sm{-I_n}{I_n}{}{I_n}$ we get
\begin{multline*}
\lambda(I(s,w)\varphi,s)=\int_{\new M_\nu\bs\new M}\beta(\varphi_s(\nu hw))\ dh=
\int_{\new M_\nu\bs\new M}\beta(\varphi_s(\nu wh))\ dh\\=
\int_{\new M_\nu\bs\new M}\beta(\varphi_s(\sm{-I_n}{I_n}{}{I_n}\nu h))\ dh=
\omega_\sigma(-1)\lambda(\varphi,s)
\end{multline*}
for all $s$. By taking the leading term in the Laurent expansion at $s=0$
we get a non-zero $\new M$-invariant functional $\lambda$ on $\pi$
such that $\lambda\circ\pi(w)=\omega_\sigma(-1)\lambda$.
Note that $\new M$ is conjugate to $H_\GLnn$ by a suitable permutation matrix which takes
$w$ to $\wnn$. We infer that $\eps_\pi=\omega_\sigma(-1)$.
Finally note that $\epsilon(\frac12,\pi,\psi)=\omega_\sigma(-1)$.
\end{proof}

Recall the following classification result due to Matringe.

\begin{theorem}[\cite{1301.0350}] \label{thm: genmetaclassification}
Suppose that $F$ is $p$-adic.
Then the set $\Irr_{\gen,\meta}\GLnn$ consists of the irreducible representations of the form
\[
\pi=\sigma_1\times\d\sigma_1\times\dots\times\sigma_k\times\d\sigma_k\times
\tau_1\times\dots\times\tau_l
\]
where $\sigma_1,\dots,\sigma_k$ are essentially square-integrable,
$\tau_1,\dots,\tau_l$ are square-integrable of metaplectic type (i.e., $L(0,\tau_i,\wedge^2)=\infty$) for all $i$.
\end{theorem}

(We expect the same result to hold in the archimedean case as well.
For an analogous result in a slightly different setup see \cite[Appendix B]{MR2930996}.)
Theorem \ref{thm: genmetaeps} now follows from Theorem \ref{thm: genmetaclassification},
Lemma \ref{lem: simpleind} and Proposition \ref{prop: metasqr}.

\begin{remark}
In fact one expects that $\Irr_{\gen,\shal}\GLnn=\Irr_{\gen,\meta}\GLnn$
(which is stronger than Theorem \ref{thm: genmetaeps}).
As was pointed to us by Wee Teck Gan, this can be proved
using the theta correspondence for the pair $(\GL_{2n}, \GL_{2n})$.
(This approach also gives a different proof for the fact that $\Irr_{\shal}\GLnn\subset\Irr_{\meta}\GLnn$.)
We omit the details since we will not be using this result.
\end{remark}

\subsection{Side remarks}\footnote{This subsection will not be used in the rest of the paper.} \label{rem: existmetacusp}
We may also consider representations of metapletic type of $\GLnn$ over a finite field $\F_q$.
(In this case, the space of $H_{\GLnn}$-invariant forms is not one dimensional in general.)
Let $A$ be an \'etale algebra of dimension $2n$ over $\F_q$, so that $T_A=A^*$ is the $\F_q$-points of a maximal torus of
$\GLnn$ over $\F_q$. Let $\lambda:T_A\rightarrow\C^*$ be a character in general position, i.e.~$\sigma(\lambda)\ne\lambda$
for any non-trivial automorphism $\sigma$ of $A$ over $\F_q$ (or equivalently a non-trivial element of the normalizer of
$T_A$ in $\GLnn$).
Let $R(A,\lambda)$ be the corresponding irreducible representation of $\GLnn(\F_q)$, which is up to a sign
the Deligne--Lusztig virtual representation $R_{T_A}^\lambda$ attached to the pair $(T_A,\lambda)$ \cite{MR0393266}.
Then it follows from \cite[Theorem 10.3]{MR1106911} that $R(A,\lambda)$ is of metaplectic type if and only if
there exists an involution $\theta$ of $A$ over $\F_q$ such that the fixed point subalgebra $A^\theta$ is of dimension $n$ over $\F_q$
and the restriction of $\lambda$ to $(A^\theta)^*$ is trivial. In particular, a cuspidal representation of $\GLnn(\F_q)$ is of metaplectic
type if and only if it is of the form $R(\F_{q^{2n}},\lambda)$ where $\lambda\rest_{\F_{q^n}^*}\equiv1$
(and of course $\lambda^{q^i}\not\equiv\lambda$ for $0<i<2n$).
Hence, cuspidal representations of $\GLnn(\F_q)$ of metaplectic type exist (e.g., one can take a character $\lambda$ of
$\F_{q^{2n}}^*$ of order $q^n+1$).

If $F$ is a $p$-adic field with residue field $\F_q$ and we inflate a cuspidal representation of $\GLnn(\F_q)$ of metaplectic type to
$Z_{\GLnn}(F)\cdot\GLnn(\OO)$ and then induce to $\GLnn(F)$ then it is easy to see that we get a supercuspidal representation of metaplectic type.
(For a much more general analysis of supercuspidal distinguished representations see \cite{MR2431732}.)

Suppose now that $F$ is $p$-adic.

In general, if $H_0$ is the fixed point subgroup of an involution of a group $G_0$
(all defined over $F$) then we say that a (not necessarily irreducible) representation $\pi$ of $G_0$ is
\emph{$H_0$-relatively cuspidal} if $\pi$ is $H_0$-distinguished, i.e. it admits a non-trivial $H_0$-invariant functional,
and for every such functional $\ell$ the function $g\mapsto\ell(\pi(g)v)$ is compactly supported modulo $H_0Z_{G_0}$ for all $v\in\pi$.

Consider the class $\Irr_{\mcusp}\GLnn$ of representations of the form
\begin{equation} \label{eq: pi}
\pi=\tau_1\times\dots\times\tau_k
\end{equation}
where $\tau_i\in\Irr_{\cusp,\meta}\GL_{2n_i}$, $i=1,\dots,k$ are pairwise inequivalent and $n=n_1+\dots+n_k$.
By Lemma \ref{lem: simpleind} we have $\Irr_{\mcusp}\GLnn\subset\Irr_{\meta}\GLnn$.

\begin{lemma} \label{lem: relsc}
The representations in $\Irr_{\mcusp}\GLnn$ are $H_\GLnn$-relatively cuspidal.
\end{lemma}

\begin{proof}
We use the results of \cite{MR2428854}.
It will be more convenient to consider the involution $\theta^{\wnn}: g\mapsto \wnn g \wnn$.
Let $H_\GLnn'$ be the fixed point subgroup of $\theta^{\wnn}$ in $\GLnn$.
Note that $\wnn$ is conjugate to $E$ in $\GLnn$ and therefore $H_\GLnn'$ is conjugate to $H_\GLnn$.
Thus, we need to show that any representation in $\Irr_{\mcusp}\GLnn$ is $H_\GLnn'$-relatively cuspidal.

In the notation of \cite{MR2428854} we take $S_0=\{\diag(t_1,\dots,t_n,t_n^{-1},\dots,t_1^{-1})\}$
(a maximal $\theta^{\wnn}$-split torus), $A_\emptyset=T_{\GLnn}$, $P_\emptyset=B_{\GLnn}$.

Let $\pi\in\Irr_{\mcusp}\GLnn$.
By \cite[Theorem~B and Lemma~2.5]{MR2428854}, it suffices to show that for any proper standard $\theta^{\wnn}$-split parabolic subgroup
$P^\circ=M^\circ\ltimes U^\circ$ of $\GLnn$, the Jacquet module $J_{P^\circ}(\pi)$ of $\pi$ with respect to $P^\circ$ is not $M^\circ\cap H_\GLnn'$ distinguished.
(By definition, a parabolic subgroup $P^\circ$ is $\theta^{\wnn}$-split if $\theta^{\wnn}(P^\circ)$ and $P^\circ$ are opposite.)
Note that for such $P^\circ$ we have $M^\circ=\diag(\GL_{n_1},\ldots,\GL_{n_r})$ with $n_i=n_{r+1-i}$ for all $i$ with $r>1$.
Let $\rho=\rho_1\otimes\dots\otimes\rho_r$ be an irreducible subquotient of $J_{P^\circ}(\pi)$. Then
there exists a function $\sigma:\{1,\dots,k\}\rightarrow\{1,\dots,r\}$ such that for any $j$,
$\rho_j=\tau_{i_1}\times\dots\times\tau_{i_l}$ where $\{i_1,\dots,i_l\}=\sigma^{-1}(j)$. Since the $\tau_i$'s are distinct and self-dual,
it is clear that we cannot have $\rho_1\cong \d{\rho_r}$. In particular, $\rho$ cannot be $M^\circ\cap H_\GLnn'$ distinguished.
Since this is true for any irreducible submodule of $J_{P^\circ}(\pi)$ we conclude that $J_{P^\circ}(\pi)$ is not $M^\circ\cap H_\GLnn'$ distinguished.
The Lemma follows.
\end{proof}

As was pointed to us by Yiannis Sakellaridis, the Steinberg representation of $\GL_2$ is also
relatively cuspidal with respect to the diagonal torus.

\begin{remark}
We already remarked before that $\Irr_{\meta,\cusp}\GLnn$ is non-empty.
However, it is even easier to show that $\Irr_{\mcusp}\GLnn$ is non-empty, since
$\Irr_{\meta,\cusp}\GL_2=\Irr_{\cusp}\PGL_2$.
\end{remark}

\subsection{Global setting}
Let now $F$ be a number field.
We say that $\pi\in\Cusp\GLnn$ is of \emph{metaplectic type} if
\[
\int_{H_\GLnn(F)\bs H_\GLnn(\A)\cap\GLnn(\A)^1}\varphi(h)\ dh\ne0
\]
for some $\varphi$ in the space of $\pi$.
Equivalently, $L^S(\frac12,\pi)\res_{s=1}L^S(s,\pi,\wedge^2)\ne0$ (\cite[Theorem 4.1]{MR1241129}).\footnote{We can
replace the partial $L$-function by the completed one since the local factors are holomorphic and non-zero.}
In particular, $\pi$ is self-dual and admits a trivial central character.
We write $\Cusp_{\meta}\GLnn$ for the set of irreducible cuspidal representations of metaplectic type.

Recall also that by the results of Jacquet--Shalika, $L^S(s,\pi,\wedge^2)$ has a pole at $s=1$
if and only if the global Shalika functional
\[
\int_{\Shal(F)\bs\Shal(\A)}\varphi(s)\psi_{\Shal}(s)^{-1}\ ds
\]
does not vanish identically on the space of $\pi$ \cite{MR1044830}. Therefore, any local component of
such $\pi$ is of Shalika type. In particular, this is the case if $\pi\in\Cusp_{\meta}\GLnn$.

We conclude

\begin{lemma}
For any $\pi\in\Cusp_{\meta}\GLnn$ we have $\pi_v\in\Irr_{\meta,\eps}\GLnn_v$ for all $v$.
\end{lemma}
Of course, in the $p$-adic case the lemma is covered by Theorem \ref{thm: genmetaeps},
but it comes by a different reasoning.

\section{Explicit local descent} \label{sec: local conjecture}

We go back to the local setup.
Let $\pi\in\Irr_{\gen}\Levi$ with Whittaker model $\WhitML(\pi)$.
For any $g\in G$ define $\wgt{g}$ by $\wgt{u\levi(m)k}=\abs{\det m}$
for any $u\in U$, $m\in\GLnn$, $k\in K$. For any $f\in C^\infty(G)$ and $s\in\C$ define $f_s(g)=f(g)\wgt{g}^s$, $g\in G$.
Let $\Ind(\WhitML(\pi))$ be the space of $G$-smooth left $U$-invariant functions $W:G\rightarrow\C$ such that
for all $g\in G$, the function $\modulus_P(m)^{-\frac12}W(mg)$ on $\Levi$ belongs to $\WhitML(\pi)$.
For any $s\in\C$ we have a representation $\Ind(\WhitML(\pi),s)$ on the space $\Ind(\WhitML(\pi))$ given by
$(I(s,g)W)_s(x)=W_s(xg)$, $x,g\in G$.

Define the intertwining operator $M(\pi,s)=M(s):\Ind(\WhitML(\pi),s)\rightarrow\Ind(\WhitML(\d\pi),-s)$ by
(the analytic continuation of)
\begin{equation} \label{eq: defM}
M(s)W(g)=\wgt{g}^s\int_U W_s(\levi(\spclt) w_U ug)\,du
\end{equation}
where $\spclt=E$ is introduced in order to preserve the character $\psi_{N_\Levi}$.
By abuse of notation we will also denote by $M(\pi,s)$ the intertwining operator
$\Ind(\WhitMLd(\pi),s)\rightarrow\Ind(\WhitMLd(\d\pi),-s)$ defined in the same way.

In the case where $F$ is $p$-adic with $p$ odd and $\pi$ and $\psi$ are unramified we have
\begin{equation} \label{eq: unramwhit}
M(s)W^\circ=\frac{L(s,\pi)}{L(s+1,\pi)}\frac{L(2s,\pi,\wedge^2)}{L(2s+1,\pi,\wedge^2)}W'^\circ
\end{equation}
for the $K$-fixed elements $W^\circ\in\Ind(\WhitML(\pi))$, $W'^\circ\in\Ind(\WhitML(\d\pi))$ such that
$W^\circ(e)=W'^\circ(e)=1$.

\subsection{}

Next we prove the following
\begin{proposition} \label{prop: M1/2}
Suppose that $\pi\in\Irr_{\gen}\GL_m$ is self-dual. Then $M(\pi,s)$ is holomorphic at $s=\frac12$.
\end{proposition}

We will need a few auxiliary results.
The first one is a well-known irreducibility result for $\GL_m$.
It is a special case of \cite[Theorem 5.1 and Proposition 5.3]{MR1634020} (cf.~also \cite[\S I.4]{MR1026752} and the references therein).

\begin{lemma} \label{lem: irred GL_n}
Let $\pi_i\in\Irr_{\sqr}\GL_{n_i}$, $i=1,2$. Suppose that the standard intertwining operator
\[
\pi_1\abs{\det\cdot}^{s_1}\times\pi_2\abs{\det\cdot}^{s_2}\rightarrow\pi_2\abs{\det\cdot}^{s_2}\times\pi_1\abs{\det\cdot}^{s_1}
\]
has a pole at $(s_1,s_2)$. Then $L(s_1-s_2,\pi_1\otimes\d\pi_2)=\infty$ and therefore
$\pi_1\abs{\det\cdot}^{s_1-\frac12}\times\pi_2\abs{\det\cdot}^{s_2+\frac12}$ is reducible.
If moreover $\pi_1$, $\pi_2$ are supercuspidal then $\pi_1\abs{\det\cdot}^{s_1}=\pi_2\abs{\det\cdot}^{s_2}$.
\end{lemma}

\begin{lemma} \label{lem: sqr-int inter}
Let $\sigma\in\Irr_{\sqr}\GL_m$.
Then the poles of $M(\sigma,s)$ are contained in those of $L(2s,\sigma\otimes\sigma)$.
Thus if $s$ is a pole then $\sigma\abs{\det\cdot}^{s-\frac12}\times\d\sigma\abs{\det\cdot}^{\frac12-s}$
is reducible.
\end{lemma}

\begin{remark}
Following \cite{MR1634020} one expects that under the condition of Lemma \ref{lem: sqr-int inter} we have the stronger conclusion that
$L(s,\sigma)^{-1}L(2s,\sigma,\wedge^2)^{-1}M(\sigma,s)$ is holomorphic.
The results of \cite{MR2581039} may be relevant for this question.
In any case, we will give a direct argument.
\end{remark}

\begin{proof}
The statement of the lemma is clearly invariant under twisting by an unramified character.
Therefore, we may assume that $\sigma$ is square-integrable.
Consider the $p$-adic case first.
Once again, upon twisting $\sigma$ by $\abs{\det\cdot}^{\iii t}$, it is enough to consider the poles on the real line.
If $\sigma$ is supercuspidal then it is known that the only possible real pole of $M(\sigma,s)$
is at $s=0$ and it can only occur if $\sigma$ is self-dual, i.e. if $L(0,\sigma\otimes\sigma)=\infty$.
Now let $\sigma$ be an arbitrary square-integrable representation.
Once again, if $\sigma$ is not self-dual then $M(\sigma,s)$ is holomorphic on $\R$.
Suppose therefore that $\sigma$ is self-dual.
By Bernstein--Zelevinsky we can realize $\sigma$ as the unique irreducible subrepresentation of
$\Sigma=\rho\abs{\det\cdot}^l\times\dots\times\rho\abs{\det\cdot}^{-l}$
for some non-negative half integer $l$ such that $2l+1$ is a divisor of $m$ and a supercuspidal representation $\rho$ on $\GL_{m/(2l+1)}$.
Moreover, $\rho$ is self-dual since $\sigma$ is.
The operator $M(\sigma,s)$ is the restriction of $M(\Sigma,s)$.
We decompose $M(\Sigma,s)$ as the composition of co-rank one intertwining operators.
Thus, the poles of $M(\Sigma,s)$ are contained in the union of the poles of
\[
M(\rho,s+k),\ \ k\in\{-l,\dots,l\}
\]
and those of the intertwining operators
\[
\rho\abs{\det\cdot}^{i+s}\times\rho\abs{\det\cdot}^{-j-s}\rightarrow
\rho\abs{\det\cdot}^{-j-s}\times\rho\abs{\det\cdot}^{i+s}, \ \ i,j\in\{-l,\dots,l\}, i\ne j.
\]
The real poles are therefore contained in the half integers between $\pm l$.
On the other hand, we know that $M(\sigma,s)$ is holomorphic for $\Re s>0$.
Thus, the real poles of $M(\sigma,s)$ are contained in the half integers between
$-l$ and $0$. These are exactly the real poles of $L(2s,\sigma\otimes\sigma)$.

Consider now the case $F=\R$.
Suppose first that $m=1$ and $\sigma$ is a character. Twisting by $\abs{\cdot}^{\iii t}$ we may assume that
either $\sigma=1$ or $\sigma=\sgn$.
The poles of $M(\sigma,s)$ on $\C$ are $2\Z_{\le0}$ in the former case
and $-1-2\Z_{\le0}$ in the latter.
In both cases the poles of $L(2s,\sigma\otimes\sigma)$ are $\Z_{\le0}$.

Suppose now that $m=2$ and $\sigma$ is a square-integrable representation of $\GL_2(F)$.
Up to a twist by $\abs{\det\cdot}^{\iii t}$ we may assume that $\sigma$ is
the irreducible subrepresentation of
$\Sigma=\abs{\cdot}^l\times\sgn^{2l-1}\abs{\cdot}^{-l}$ for some positive half-integer $l$.
Once again by decomposing $M(\Sigma,s)$ into a product, the poles
of $M(\sigma,s)$ are contained in those of
$M(\mathbf{1}_{\R^*},s+l)$, $M(\sgn^{2l-1},s-l)$ and
$\abs{\cdot}^{s+l}\times\sgn^{2l-1}\abs{\cdot}^{-s+l}\rightarrow
\sgn^{2l-1}\abs{\cdot}^{-s+l}\times\abs{\cdot}^{s+l}$.
Thus, the poles of $M(\Sigma,s)$ on $\C$ consist of half integers,
and therefore the poles of $M(\sigma,s)$ are non-positive half integers.
On the other hand, the poles of $L(2s,\sigma\otimes\sigma)$ are the non-positive half-integers.

Finally, if $F=\C$ we may assume that $\sigma$ is the unitary character $(z/\bar z)^l$,
for some half integer $l$.
The poles of $M(\sigma,s)$ on $\C$ are $-\abs{l}+\Z_{\le0}$ while the poles
of $L(2s,\sigma\otimes\sigma)$ are $-\abs{l}+\frac12\Z_{\le0}$.
\end{proof}

\begin{proof}[Proof of Proposition \ref{prop: M1/2}]
Write
\[
\pi=\sigma_1\times\dots\times\sigma_k
\]
where $\sigma_1,\dots,\sigma_k$ are essentially square-integrable.
We remark that since $\pi$ is self-dual and irreducible, $\sigma_i\times\d\sigma_j$ is irreducible
for all $i,j$.
By decomposing $M(\pi,s)$ into a product of co-rank one intertwining operators
it suffices to show that the following intertwining operators
\begin{gather*}
\sigma_i\abs{\cdot}^s\times\d\sigma_j\abs{\cdot}^{-s}\rightarrow\d\sigma_j\abs{\cdot}^{-s}\times\sigma_i\abs{\cdot}^s,
\ \ i,j=1,\dots,k, i\ne j,\\
M(\sigma_i,s),\ \ i=1,\dots,k,
\end{gather*}
are holomorphic at $s=\frac12$.
This follows from Lemma \ref{lem: irred GL_n} (resp., Lemma \ref{lem: sqr-int inter})
and the irreducibility of $\sigma_i\times\d\sigma_j$ (resp. $\sigma_i\times\d\sigma_i$).
\end{proof}

\subsection{Local Fourier--Jacobi transform -- $p$-adic case}
In this subsection $F$ is $p$-adic.
Following \cite{MR1675971}, for any $W\in\Ind(\WhitML(\pi))$ and $\Phi\in \swrz(F^n)$ define a genuine function on $\tilde G$:
\begin{equation}\label{eq: defwhitform}
\whitform(W,\Phi,\animg g,s)=\int_{V_\gamma\bs V} W_s(\gamma v \toG(g))\wevinv(v\animg g)\Phi(\xi_n)\,dv,\,\,\,g\in G'
\end{equation}
where the element $\gamma\in G$ and the subgroup $V_\gamma$ were defined in \S\ref{sec: elements}.

The basic properties of $\whitform$ are summarized in the following lemma.
Note that the first part is not valid for the more general integrals introduced in \cite{MR1675971}.

\begin{lemma} \label{L: whitform}
\begin{enumerate}
\item The integrand in \eqref{eq: defwhitform} is compactly supported uniformly in $s$.
Thus $s\mapsto\whitform(W,\Phi,\tilde g,s)$ is entire.
\item For any $W\in\Ind(\WhitML(\pi))$, $\Phi$ and $s\in\C$, the function
$\tilde g\mapsto\whitform(W,\Phi,\tilde g,s)$ is smooth and $(\tilde N,\psi_{\tilde N})$-equivariant.
\item For any $\tilde g\in\tilde G$, $v\in V$ and $x\in G'$ we have
\begin{equation} \label{eq: whitform equivariance}
\whitform(I(s,v\toG (x))W,\wevinv(v\animg x)\Phi,\tilde g,s)=
\whitform(W,\Phi,\tilde g\animg x,s).
\end{equation}
\item \label{part: unramwhitform} Suppose that $p\ne2$, $\pi$ is unramified, $\psi$ is unramified, $W^\circ \in\Ind(\WhitML(\pi))$
is the standard unramified vector, and $\Phi_0=1_{\OO^n}$. Then $\whitform(W^\circ,\Phi_0,e,s)\equiv1$.
\end{enumerate}
\end{lemma}

\begin{proof}
We first note that the integrand defining $\whitform$ is left $V_\gamma$-invariant
since for any $v\in V_\gamma$ we have
\[
\wevinv(v)\Phi(\xi_n)=\psi_V(v)^{-1}\psi(v_{n,2n+1})^{-1}\Phi(\xi_n)
\]
(by \eqref{eq: weilH2} and \eqref{eq: weilext}) and
$\psi_N(\gamma v\gamma^{-1})=\psi_V(v)\psi(v_{n,2n+1})$.

It is enough to prove the first statement for $g=e$.
Conjugating by $\gamma$, using \eqref{eq: weilH1}, \eqref{eq: weilH3} and \eqref{eq: weilext},
the integral in \eqref{eq: defwhitform} becomes
\begin{equation}\label{eq: convaux}
\int W_s(\toUbar(u)\gamma)\Phi(\xi_n-\bar x)\psi(\frac12u_{n,n+1})\,du,
\end{equation}
where the integration is over $u$ of the form $u=\sm{x}{y}{0}{\startran{x}}$ (a vector subspace of $\symspace_{2n}$)
and $\bar x$ is the bottom row of $x$.
Let $a=\levi(\diag(t_1^{-1},\ldots,t_{2n}^{-1}))\in T$ be the diagonal part of the Iwasawa decomposition of
$\toUbar(u)$. Let $\alpha=\max(\abs{t_1/t_2},\dots,\abs{t_{2n-1}/t_{2n}})$
and $\beta=\max(1,\abs{t_2/t_1},\dots,\abs{t_{2n}/t_{2n-1}})$.
Note that $\abs{t_1\dots t_{2n}}\ge1$.
We will see below that
\begin{equation} \label{eq: maineq}
\log\max(\norm{u},\abs{t_1\dots t_{2n}},\alpha)\ll\log\max(\norm{\bar x},\beta).
\end{equation}
On the other hand, by the analytic properties of Whittaker functions, $W_s(ak)=0$ for $k\in K$ unless $\beta$ is bounded above
(in terms of $W$).
The first part of the lemma would therefore follow from \eqref{eq: maineq}.

To show \eqref{eq: maineq}, let $s_i=\prod_{j=1}^it_j$, $i=1,\dots,2n$.
Observe that $s_i$ is the $L^\infty$-norm of the vector $\minors_i$ in $F^{4n\choose i}$ whose
coordinates are the $i\times i$ minors of the last $i$ rows of $\toUbar(u)$.
Note that the nonzero coordinates of $\minors_n$ (say $a_1,\dots,a_r$) are the minors of size $\le n$ of $x$
and that among the coordinates in $\minors_{2n}$ are all products $a_ia_j$.
It follows that $s_n^2\le s_{2n}$, i.e. $\abs{t_1\dots t_n}\le\abs{t_{n+1}\dots t_{2n}}$.
This easily implies that $\log\alpha\ll\log\beta$.
Next we claim that
\begin{equation} \label{eq: sn}
s_ns_{n+1}\le\max(1,\norm{\bar x})s_{2n}
\end{equation}
i.e.,
\begin{equation} \label{eq: tis2}
\abs{t_1\dots t_n}\le\max(1,\norm{\bar x})\abs{t_{n+2}\dots t_{2n}}
\end{equation}
Indeed, the nonzero coordinates of $\minors_{n+1}$ are either
the minors (say $b_1,\dots,b_s$) of the right bottom corner of $\toUbar(u)$ of size $(n+1)\times 3n$ or
the entries of $\bar x$ times the $a_i$'s, while any $a_i$ is an alternating sum of the entries of $\norm{\bar x}$
(or $1$) times the minors (say $c_1,\dots,c_t$) of the top $n-1$ rows of $x$.
Since the products $b_ic_j$ and $a_ia_j$ are among the coordinates of $\minors_{2n}$ the inequality
\eqref{eq: sn} follows.

It follows now from \eqref{eq: tis2} that up to a multiplicative constant, $\abs{t_1}$, and hence
all $\abs{t_j}$'s, are bounded by a suitable power of $\max(\norm{\bar x},\beta)$. Thus,
\[
\log\abs{s_{2n}}\ll\log\max(\norm{\bar x},\beta).
\]
On the other hand $\norm{u}\le\abs{s_{2n}}$ since the entries of $u$ are among the coordinates of $\minors_{2n}$.
The relation \eqref{eq: maineq} follows.

The fact that $\whitform(W,\Phi,\cdot,s)$ is a Whittaker function on $\tilde G$
follows from the facts that $\toG(N')$ normalizes both $V$ and $V_\gamma$ and that for any $n\in N'_{\Levi'}$ and $u\in U'$ we have
$\psi_N(\gamma \toG(nu)\gamma^{-1})=\psi_{N'_{\Levi'}}(n)$
and $\wevinv(\animg{n}\animg{u})\Phi(\xi_n)=\psi_{U'}(u)\Phi(\xi_n)$
(by \eqref{eq: weil1} and \eqref{eq: weil3}).
It is also clear from \eqref{eq: weilext2} that the map $\whitform$ is equivariant.

Finally, in the unramified case, the right-hand side of \eqref{eq: maineq} equals $0$ on the
support of $W^\circ$ and therefore, the integrand in \eqref{eq: convaux} is supported in the set $\norm{u}\le1$.
\end{proof}

For an essentially square-integrable representation $\delta$ of $\GL_m$ let $\expo(\delta)$
be the real number $\alpha$ such that $\delta\abs{\det}^{\alpha}$ is unitary.\footnote{Note the twist by $\abs{\det}^{\alpha}$
rather than $\abs{\det}^{-\alpha}$}
Suppose that $\pi\in\Irr_{\gen}\GLnn$. We can write $\pi=\delta_1\times\dots\times\delta_k$
where $\delta_i$ are essentially square-integrable and are uniquely determined up to permutation.
Set $\expo(\pi)=\max_i\expo(\delta_i)$. Thus, $\expo(\pi)\in [0,\frac12)$ if $\pi$ is unitarizable
and in this case $\expo(\pi)=0$ if and only if $\pi$ is tempered.

Next, let us recall another result of Ginzburg--Rallis--Soudry.
Let $V_1$ be the inverse image of the center of the Heisenberg group under the quotient map $V\rightarrow V/V_0$.
Let $\psi_{V_1}$ be the character of $V_1$ such that $\wev(v)$ acts by the scalar $\psi_{V_1}(v)$ for $v\in V_1$.
Let $I(\pi)^\circ$ be the (vector) subspace of $I(\pi)$ consisting of sections which are supported in
the open set $G^\circ=P\gamma V\toG(G')$.
Denote by $J_{(V_1,\psi_{V_1})}$ the twisted Jacquet module (i.e., the co-invariants)
with respect to $(V_1,\psi_{V_1})$. By \cite[\S 6.1]{MR1671452} the canonical map
\begin{equation} \label{eq: restmap}
J_{V_1,\psi_{V_1}}(I(\pi)^\circ)\rightarrow J_{V_1,\psi_{V_1}}(I(\pi))
\end{equation}
is an isomorphism.

\begin{lemma} \label{lem: Aasymp}
Suppose that $\pi\in\Irr_{\gen}\GLnn$ and $\alpha\in\R$ with $\alpha>\expo(\pi)$.
Then for any $W\in\Ind(\WhitML(\pi))$ and $\Phi\in \swrz(F^n)$ there exists $c>0$ such that
for any $t'=\levi'(t)\in T'$, $k\in K'$, $s\in\C$ we have
$\whitform(W,\Phi,\animg{t'}\animg k,s)=0$ unless $\abs{t_i}\le c\abs{t_{i+1}}$ for $i=1,\dots,n$
(where $t=\diag(t_1,\dots,t_n)$ and $t_{n+1}=1$) in which case
\[
\abs{\whitform(W,\Phi,\animg{t'}\animg k,s)}\ll
\modulus_{B'}(t')^{\frac12}\abs{\det(t)}^{\Re s+\frac12-\alpha}.
\]
\end{lemma}

\begin{proof}
Using \eqref{eq: restmap} (cf.~\cite[Corollary 2.2]{ST}) we can find $W'\in I(\pi)^\circ$ such that
$\whitform(W,\Phi,\animg{g},s)=\whitform(W',\Phi,\animg{g},s)$ for all $g\in G'$, $s\in\C$.
Thus, we may assume without loss of generality that $W\in I(\pi)^\circ$. Also, it suffices to deal with the case $k=e$.
We write
\[
\whitform(W,\Phi,\animg{t'},s)=\int_{V_\gamma\bs V} W_s(\gamma v \toG(t'))\wevinv(v\animg{t'})\Phi(\xi_n)\,dv.
\]
Making a change of variable $v\mapsto\toG(t')v\toG(t')^{-1}$ we get (using \eqref{eq: weil1})
\begin{multline*}
\abs{\det t}^{-n}\int_{V_\gamma\bs V} W_s(\toLevi(t)\gamma v)\wevinv(\animg{t'}v)\Phi(\xi_n)\,dv=\\
\beta_\psi(t')\abs{\det t}^{\frac12-n}\int_{V_\gamma\bs V} W_s(\toLevi(t)\gamma v)\wevinv(v)\Phi(t_n\xi_n)\,dv.
\end{multline*}
By our condition on $W$, the integral over $v$ is compactly supported independently of $t$ and $s$.
Thus, up to changing $W$ and $\Phi$ it suffices to consider the term $v=e$, i.e.,
\[
\beta_\psi(t')\abs{\det t}^{\frac12-n}W_s(\toLevi(t))\Phi(t_n\xi_n).
\]
Note that $\modulus_{B'}(t')^{\frac12}=\modulus_B(\toLevi(t))^{\frac12}\abs{\det t}^{-n}$.
Thus, the lemma follows from standard estimates on the Whittaker function (e.g., \cite[Lemma 2.1]{LMao4}).
\end{proof}

\begin{remark}
More precisely, the proof of the lemma gives the precise asymptotics of $\whitform(W,\Phi,\animg{t'}\animg k,s)$
in terms of the exponents of $\pi$ (using \cite{MR2495561}).
\end{remark}

\subsection{Local Fourier--Jacobi transform -- archimedean case}
Now assume that $F$ is archimedean.
We start with a standard estimate on Whittaker functions.
\begin{lemma}
Let $\pi\in\Irr_{\gen}\GLnn$.
For any $m\ge0$ and $\alpha>\expo(\pi)$ there exists a continuous seminorm $\lambda$ of $\WhitM(\pi)$ such that
for any $W\in \WhitM(\pi)$ we have
\begin{equation} \label{eq: whitestim}
\abs{W(tk)}\le\lambda(W)\prod_{i=1}^{2n-1}(1+t_it_{i+1}^{-1})^{-m}
\modulus_{B_\GLnn}(t)^{\frac12}\abs{\det(t)}^{-\alpha},
\end{equation}
where $t=\diag(t_1,\dots,t_{2n})\in T_\GLnn$ with $t_{2n}=1$ and $k\in K_\GLnn$.
\end{lemma}

\begin{proof}
For any $f\in C_c(\GLnn)$ we have
\begin{multline*}
(f*W)(tk)=\int_{\GLnn}f(g)W(tkg)\ dg=\int_{\GLnn}f(k^{-1}g)W(tg)\ dg\\=\int_{N_\GLnn\bs\GLnn}\int_{N_\GLnn}f(k^{-1}ng)W(tng)\ dn\ dg=
\int_{N_\GLnn\bs\GLnn}\int_{N_\GLnn}f(k^{-1}ng)\psi_N(tnt^{-1})W(tg)\ dn\ dg\\=
\int_{N_\GLnn\bs\GLnn}\phi_{k,g}(t_1t_2^{-1},\dots,t_{2n-1}t_{2n}^{-1})W(tg)\ dg
\end{multline*}
where $\phi_{g_1,g_2}$ is the Fourier transform of $\int_{N_\GLnn^{\der}}f(g_1^{-1}\cdot ng_2)\ dn\in C_c(N_\GLnn/N_\GLnn^{\der})$.
(Note that we can integrate over $g$ in a compact subset of $N_\GLnn\bs\GLnn$ depending only on the support of $f$.)
It follows from a standard argument (using \cite[Lemma 4.1]{MR518111}) that there exists a compact subset $C\subset\GLnn$
and for any $m\ge1$ there exists $X_m$ in the universal enveloping algebra of the Lie algebra of $\GLnn$ such that
\[
\abs{W(tk)}\le\prod_{i=1}^{2n-1}(1+\abs{t_it_{i+1}^{-1}})^{-m}\sup_{c\in C}\abs{X_mW(tc)}.
\]
By the uniform moderate growth of $W$, it remains to prove \eqref{eq: whitestim} in the chamber $\abs{t_1}<\dots<\abs{t_{2n}}$.
This follows from \cite[Theorem 15.2.5 and Lemma 15.2.3]{MR1170566}.
\end{proof}

This immediately gives the following analogue of Lemma \ref{L: whitform}.
(The proof is similar and will be omitted.)
\begin{lemma} \label{L: whitformarch}
For any $W\in\Ind(\WhitML(\pi))$ and $\Phi\in\swrz(F^n)$ the integral \eqref{eq: defwhitform} is well defined and
absolutely convergent uniformly for $s$ and $g$ in compact sets. Thus $\whitform(W,\Phi,\tilde g,s)$ is entire in $s$
and as a function of $\tilde g$ is smooth, $(\tilde N,\psi_{\tilde N})$-equivariant
and of moderate growth (modulo $\tilde N$). Finally \eqref{eq: whitform equivariance} holds.
\end{lemma}

Finally, we will prove an analogue of Lemma \ref{lem: Aasymp}.

For a smooth Fr\'echet representation $\pi$ of $M$ let $I(\pi)^\circ$ denote the closed subspace of $I(\pi)$ consisting of
sections which vanish together with all their derivatives on the complement of $G^\circ$.
Also, if $(\sigma,E_\sigma)$ is representation of $V_1$ on a topological vector space, we set
$J_{(V_1,\psi_{V_1})}(\sigma)=\sigma/\overline{\sigma_{V_1,\psi_{V_1}}}$
where $\sigma_{V_1,\psi_{V_1}}$ is the span of $\sigma(v)m-\psi_{V_1}(v)m$, $v\in V_1$, $m\in E_\sigma$.

\begin{lemma} \label{lem: archsurj}
Let $(\pi,E_\pi)$ be a smooth Fr\'echet representation of $M$ of moderate growth (\cite[\S11.5.1]{MR1170566}).
Then the map \eqref{eq: restmap} is surjective.
\end{lemma}

\begin{proof}
Let $\swrz(U\bs G;E_\pi)$ be the space of Schwartz functions on $U\bs G$ with values in $E_\pi$.\footnote{For general facts
about Schwartz spaces see \cite{MR1100992} and \cite{MR2418286}.}
We can identify $\swrz(U\bs G;E_\pi)$ with the projective tensor product $\swrz(U\bs G)\hat\otimes E_\pi$.
Let $p_\pi:\swrz(U\bs G;E_\pi)\rightarrow I(\pi)$ be given by $f\mapsto\int_M\modulus_P(m)^{-\frac12}\pi(m^{-1})f(mg)\ dm$.
(This is well-defined and continuous since $\pi$ is of moderate growth.)
Then $p_\pi$ is a strict surjective $G$-intertwining homomorphism. (In fact, already the restriction of $p_\pi$ to $C_c^\infty(U\bs G;E_\pi)$ is onto.)
Now, it follows from \cite[Theorem 2.2.15]{MR3022961} that $\swrz(U\bs G)=\swrz(U\bs G^\circ)+\swrz(U\bs G)_{V_1,\psi_{V_1}}$
where we view $\swrz(U\bs G^\circ)$ as the closed subspace of $\swrz(U\bs G)$ consisting of functions which vanish together with all their
derivatives on the complement of $G^\circ$. To check that the conditions of [ibid.] are satisfied we only need to note that
$\psi_{V_1}$ is non-trivial on $V_1^x:=V_1\cap x^{-1}Ux$
for any $x\notin G^\circ$ -- a fact which was observed in \cite[\S 6.1]{MR1671452}. Indeed, this will imply that
$(\mathfrak{F}\otimes\psi_{V_1}^{-1})^{V_1^x}=0$ for any algebraic (finite-dimensional) representation $\mathfrak{F}$ of $V_1^x$ (since $\mathfrak{F}(v)$ acts
unipotently for any $v\in V_1^x$).
Thus, $\swrz(U\bs G;E_\pi)=\swrz(U\bs G^\circ;E_\pi)+\swrz(U\bs G)_{V_1,\psi_{V_1}}\hat\otimes E_\pi$
\cite[Ch.~I, \S1, no.~2, Prop.~3]{MR0075539}. Applying $p_\pi$ we get the required assertion
since $p_\pi$ maps $\swrz(U\bs G^\circ;E_\pi)$ into $I(\pi)^\circ$.
\end{proof}

\begin{lemma} \label{lem: Aasymparch}
Suppose that $\pi\in\Irr_{\gen}\GLnn$ and $\alpha\in\R$ with $\alpha>\expo(\pi)$.
Then for any $W\in\Ind(\WhitML(\pi))$, $\Phi\in \swrz(F^n)$, a compact set $D\subset\C$ and $m\ge1$ we have
\[
\abs{\whitform(W,\Phi,\animg{t'}\animg k,s)}\ll_{W,\Phi,D,m}
\modulus_{B'}(t')^{\frac12}\abs{\det(t)}^{\Re s+\frac12-\alpha}
\prod_{i=1}^n(1+\abs{t_it_{i+1}^{-1}})^{-m}
\]
for all $t'=\levi'(t)\in T'$, (with $t=\diag(t_1,\dots,t_n)$ and $t_{n+1}=1$), $k\in K'$, $s\in D$.
\end{lemma}

\begin{proof}
Let $M^1=\{\levi(m):\abs{\det m}=1\}$ so that $M=M^1\times A_M$ where $A_M=\{\levi(tI_{2n}):t\in\R_{>0}\}$.
Let $\pi^1$ be the restriction of $\pi$ to $M^1$ realized on $\WhitML(\pi^1)=\{W\rest_{M^1}:W\in\WhitML(\pi)\}$.
Consider the smooth Fr\'echet representation $\Pi=\swrz(A_M)\hat\otimes\WhitML(\pi^1)$ of $M$ which we identify with the space of
of functions $W':M\rightarrow\C$ such that for all $a\in A_M$, $W'[a]:m\mapsto W'(am)$
belongs to $\WhitML(\pi^1)$ and $a\mapsto W'[a]$ is a Schwartz function on $A_M$ with values in $\WhitML(\pi^1)$.
Likewise, we can identify $I(\Pi)$ with the space of left $U$-invariant functions $W':G\rightarrow\C$ such that $m\mapsto W'(mg)$ belongs to $\WhitML(\pi^1)$
for all $g\in G$ and $(a,k)\mapsto W'(\cdot ak)$ is a Schwartz function on $A_M\times K$ with values in $\WhitML(\pi^1)$.
For any $s\in\C$ define a continuous surjective $G$-intertwining map $\mathcal{F}_s:I(\Pi)\rightarrow I(\pi,s)$ by
\[
\mathcal{F}_s(W')(g)=\int_{A_M}\modulus_P(a)^{-\frac12}W'_{-s}(ag)\ da.
\]
Thus, $\whitform(\mathcal{F}_s(W'),\Phi,\cdot,s)$ depends only on the image of $W'$ in $J_{V_1,\psi_{V_1}}(I(\Pi))$.
Moreover, $\mathcal{F}_s(I(\Pi)^\circ)\subset I(\pi)^\circ$ and $s\mapsto\mathcal{F}_s$ is a continuous map
(in the topology of bounded convergence).
Given $W\in I(\pi)$ let $W'\in I(\Pi)$ be such that $\mathcal{F}_0(W')=W$. Then $W=\mathfrak{F}_s(W'_s)$ for all $s\in\C$.
By Lemma \ref{lem: archsurj} the map
\[
I(\Pi)^\circ\xrightarrow{j} J_{V_1,\psi_{V_1}}(I(\Pi))
\]
is surjective, and hence there exists a continuous (not necessarily linear) map $\mathfrak{s}:I(\Pi)\rightarrow I(\Pi)^\circ$
such that $j\circ\mathfrak{s}$ is the projection $I(\Pi)\rightarrow J_{V_1,\psi_{V_1}}(I(\Pi))$ (\cite[Ch. II, \S4, no.~7, Proposition 12]{MR910295}).
Thus, $\whitform(W,\Phi,\cdot,s)=\whitform(\mathfrak{F}_s(\mathfrak{s}(W'_s)),\Phi,\cdot,s)$.
The same argument as in Lemma \ref{lem: Aasymp} (using \eqref{eq: whitestim}) yields the proof.
\end{proof}

\subsection{Local integrals}

Now let $F$ be either $p$-adic or archimedean, $\pi\in\Irr_{\gen}\Levi$ and $\tilde\sigma\in\Irr_{\gen,\psi_{\tilde N}^{-1}}\tilde G$
with Whittaker model $\WhitGd(\tilde\sigma)$.
Following Ginzburg--Rallis--Soudry \cite{MR1675971}, for any $\tilde W\in \WhitGd(\tilde\sigma)$, $W\in\Ind(\WhitML(\pi))$ and $\Phi\in\swrz(F^n)$ define
the local Shimura type integral
\begin{equation}\label{eq: localinner}
\tilde{J}(\tilde W,W,\Phi,s):=\int_{N'\bs G'}\tilde W(\tilde g)\whitform(W,\Phi,\tilde g,s)\ dg.
\end{equation}
The analytic properties of this integral were worked out by Ginzburg-Rallis-Soudry \cite[\S6.3]{MR1675971}, \cite{MR1671452} and subsequently by Kaplan
\cite{Kaplan}.
In particular, $\tilde{J}$ converges in some right-half plane (depending only on
$\pi$ and $\tilde\sigma$) and admits a meromorphic continuation in $s$.
Moreover, for any $s\in\C$ we can choose $\tilde W$, $W$ and $\Phi$ such that $\tilde{J}(\tilde W,W,\Phi,s)\ne0$.
Finally, we have local functional equations
\begin{equation} \label{eq: local functional equation}
\tilde{J}(\tilde W,M(s)W,\Phi,-s)=\omega_\pi((-1)^n 2) \abs{2}^{2ns}\frac{\gamma(\tilde\sigma\otimes\pi,s+\frac12,\psi)}
{\gamma(\pi,s,\psi)\gamma(\pi,\wedge^2,2s,\psi)}
\tilde{J}(\tilde W,W,\Phi,s).
\end{equation}
We note once again that our choice of characters differs from that of Ginzburg-Rallis-Soudry.

By \eqref{eq: whitform equivariance}, for any $ x\in G'$ we have
\begin{equation} \label{eq: invariance of J}
\tilde{J}(\tilde\sigma(\animg{x})\tilde W,I(s,\toG (x))W,\weinv(\animg x)\Phi,s)=\tilde{J}(\tilde W,W,\Phi,s).
\end{equation}
In particular, at any point of holomorphy $s$, for any $W$,
\begin{equation} \label{eq: Jinv}
\text{the functional $\tilde W\mapsto\tilde{J}(\tilde W,W,\Phi,s)$ belongs to the contragredient of $\WhitGd(\tilde\sigma)$.}
\end{equation}

If $F$ is $p$-adic and $\tilde\sigma$ is supercuspidal then the integral defining $\tilde{J}(\tilde W,W,\Phi,s)$ is
absolutely convergent (in fact compactly supported) and hence $\tilde{J}(\tilde W,W,\Phi,s)$ is entire.

More generally, we have:
\begin{lemma} \label{lem: Jholom}
Suppose that $\tilde\sigma$ is square-integrable and let $\alpha=\expo(\pi)$. Then there exists $\delta>0$ such that
$\tilde{J}(\tilde W,W,\Phi,s)$ is absolutely convergent locally uniformly (hence, holomorphic) for $\Re s\ge\alpha-\frac12-\delta$.
Similarly, for any $W\in\Ind(\WhitML(\pi))$, $\d{W}\in\Ind(\WhitMLd(\pi))$, $\Phi,\d{\Phi}\in\swrz(F^n)$
\[
\tilde{J}(\whitformd(\d{W},\d{\Phi},\cdot,-\frac12),W,\Phi,s)
\]
converges absolutely and locally uniformly (and hence, is holomorphic) for $\Re s>2\alpha-\frac12$.
In particular, if $\pi$ is unitarizable then $\tilde{J}(\tilde W,W,\Phi,s)$ is holomorphic at $s=\frac12$ for any
$\tilde W\in\desinv(\pi)$, $W\in\Ind(\WhitML(\pi))$ and $\Phi\in\swrz(F^n)$.
\end{lemma}

Indeed, the second part follows Lemmas \ref{lem: Aasymp} and \ref{lem: Aasymparch} using the Iwasawa decomposition in the integral defining $\tilde{J}$.
For the first part we also use \cite[Theorem 15.2.5 and Lemma 15.2.3]{MR1170566} applied to $\tilde G$ (as well as the moderate growth
of $\tilde W$).

Finally, if $F$ is $p$-adic, $p\ne 2$, $\pi$, $\tilde\sigma$ and $\psi$ are unramified, $W^\circ$ and $\Phi_0$ are as in
Lemma \ref{L: whitform} part \ref{part: unramwhitform} and $\tilde W^\circ$ is $K'$-invariant with $\tilde W^\circ(e)=1$ then
\begin{equation} \label{eq: Junram}
\tilde{J}(\tilde W^\circ,W^\circ,\Phi_0,s)=\Delta_{\Sp_n}(1)^{-1}\frac{L_{\psi^{-1}}(\tilde\sigma\otimes\pi,s+\frac12)}{L(s+1,\pi)L(2s+1,\pi,\wedge^2)}.
\end{equation}
The factor $\Delta_{\Sp_n}(1)^{-1}$ shows up because we use the Tamagawa measure.
The $L$-factor in the numerator is the one defined by Ginzburg--Rallis--Soudry in \cite{MR1675971} (with $\psi^{-1}$ replaced by $\psi_{\frac12}^{-1}$).

\subsection{Abstract and explicit local descent}
Let $\pi\in\Irr_{\gen}\GLnn$, considered also as a representation of $\Levi$ via $\levi$.
By \cite[Theorem in \S1.3]{MR1671452}, for any non-zero subrepresentation $\pi'$ of $\Ind(\WhitML(\pi))$
there exists $W\in\pi'$ and $\Phi\in\swrz(F^n)$ such that $\whitform(W,\Phi,\cdot,0)\not\equiv0$.

\begin{remark}
We point out the following inaccuracy in the proof of [loc.~cit.] (using its notation).
The group $\overline{X}(0,n)$ is not normalized by the image of $\ell^i$.
Therefore, the root killing argument in [ibid., p. 861-2] only reduces the non-vanishing of $\whitform$ on $\pi'$ to the non-vanishing of
\[
\int W(\toUbar(u))\psi(\frac12u_{n,n+1})\,du,
\]
where $u$ is integrated over the space of strictly upper triangular matrices $\oplus_{j>i} F \one^{\symspace}_{i,j}\subset \symspace$.
This space admits a filtration
\[
\symspace^{1,1}\supset\symspace^{1,2}\supset\ldots\supset \symspace^{1,n+1}=\symspace^{2,1}\supset\ldots\supset
\symspace^{2,n}=\symspace^{3,1}\supset\ldots\supset\symspace^{2n-1,1}\supset \symspace^{2n-1,2}=\symspace^{2n,1}=0
\]
where $\symspace^{k,l}$, $k=1,\dots,2n$, $l=1,\dots,n+1-[k/2]$ is the span of $\one^{\symspace}_{i,j}$ where $i+j\le 2n+1$ and either $j-i>k$
or $j-i=k$ and $i\ge l$.
In this filtration, any two consecutive spaces $\symspace^{k,l}$ and $\symspace^{k,l+1}$ differ by the one-dimensional space $F\one^{\symspace}_{l,k+l}$. Denote by
$I_{k,l}$ the above integration over the domain $\symspace^{k,l}$. Then using the same root killing argument of
[ibid], one can show that the non-vanishing of $I_{k,l}$ on $\pi'$ is equivalent to the non-vanishing of $I_{k,l+1}$.
As $I_{2n,1}$ is clearly non-vanishing, we get the non-vanishing of $I_{1,1}$, thus of $\whitform$.
Note that the presence of $\psi(\frac12 u_{n,n+1})$ plays no role in the argument.

Other than that, the proof of [ibid.] applies equally well in the Archimedean case.
\end{remark}

Assume now that $\pi\in\Irr_{\gen,\meta}\GLnn$.
By Proposition \ref{prop: M1/2} $M(s)$ is holomorphic at $s=\frac12$.
Denote by $\des(\pi)$ the space of Whittaker functions on $\tilde G$ generated by $\whitform(M(\frac12)W,\Phi,\cdot,-\frac12)$,
$W\in \Ind(\WhitML(\pi))$, $\Phi\in\swrz(F^n)$.
By the above $\des(\pi)\ne0$.

Let $\pi'$ be the image of $M(\frac12)$.
By \eqref{eq: whitform equivariance} the space $\des(\pi)$ is canonically a quotient of
the $\tilde G$-module $J_V(\pi'\otimes\wevinv)$ of $V$-coinvariant of $\pi'\otimes\wevinv$.
We view $J_V(\pi'\otimes\wevinv)$ as the ``abstract'' descent and $\des(\pi)$ as the ``explicit'' descent.

\section{A technical property} \label{sec: technical}
Let $F$ be a local field.
In \cite[\S2]{LMao5} we defined a regularized integral
\[
\stint_{N'}f(\animg{n})\psi_{\tilde N}(n)\ dn
\]
for any matrix coefficient of an irreducible representation of $\tilde G$.
In the $p$-adic case we have
\[
\stint_{N'}f(\animg{n})\psi_{\tilde N}(n)\ dn=
\int_{N_1}f(\animg{n})\psi_{\tilde N}(n)\ dn
\]
for any sufficiently large compact open subgroup $N_1$ of $N'$
(and in fact, this is valid for any $f$ which is bi-invariant under some open subgroup of $G'$ which splits trivially in $\tilde G$).
In the archimedean case, we gave an ad-hoc definition.

Let $\pi\in\Irr_{\gen,\meta}\GLnn$.
We say that $\pi$ is \emph{\good}\ if the following conditions are satisfied for all $\psi$:
\begin{enumerate}
\item $\des(\pi)$ is irreducible.
\item $\tilde{J}(\tilde W,W,\Phi,s)$ is holomorphic at $s=\frac12$ for any $\tilde W\in\desinv(\pi)$, $W\in\Ind(\WhitML(\pi))$ and $\Phi\in\swrz(F^n)$.
(By Lemma \ref{lem: Jholom} this is automatic if $\pi$ is unitarizable.)
\item For any $\tilde W\in\desinv(\pi)$,
\begin{equation} \label{eq: factorsthru}
\tilde{J}(\tilde W,W,\Phi,\frac12)\text{ factors through the map }
(W,\Phi)\mapsto(\whitform(M(\frac12)W,\Phi,\cdot,-\frac12)).
\end{equation}
In other words (by \eqref{eq: invariance of J}) there is a non-degenerate $\tilde G$-invariant pairing $[\cdot,\cdot]$ on $\desinv(\pi)\times\des(\pi)$ such that
\[
\tilde{J}(\tilde W,W,\Phi,\frac12)=[\tilde W,\whitform(M(\frac12)W,\Phi,\cdot,-\frac12)]
\]
for any $\tilde W\in\desinv(\pi)$, $W\in\Ind(\WhitML(\pi))$ and $\Phi\in\swrz(F^n)$.
\end{enumerate}

In the next section we will show that this condition is satisfied for the local components of certain automorphic representations.

Suppose that $\pi\in\Irr_{\gen,\meta}\GLnn$ is good and let $\tilde\pi=\desinv(\pi)$.
Then by \cite[\S2]{LMao5} there exists a non-zero constant $c_\pi$ such that for any $\tilde W\in\desinv(\pi)$, $W\in\Ind(\WhitML(\pi))$ and $\Phi\in\swrz(F^n)$ we have
\begin{equation} \label{eq: main}
\stint_{N'}\tilde{J}(\tilde\pi(\animg n)\tilde W,W,\Phi,\frac12)\psi_{\tilde N}( n)\,dn
=c_\pi\tilde W(e)\whitform(M(\frac12)W,\Phi,e,-\frac12).
\end{equation}
In other words, for any $W\in\Ind(\WhitML(\pi))$, $\d{W}\in\Ind(\WhitMLd(\pi))$, $\Phi,\d{\Phi}\in\swrz(F^n)$
\begin{multline} \label{eq: main1}
\stint_{N'}\tilde{J}(\whitformd(M(\frac12)I(\frac12,\toG(n))\d{W},\we(\animg{n})\d{\Phi},\cdot,-\frac12),W,\Phi,\frac12)\psi_{\tilde N}( n)\,dn
\\=c_\pi\whitformd(M(\frac12)\d{W},\d{\Phi},e,-\frac12)\whitform(M(\frac12)W,\Phi,e,-\frac12).
\end{multline}

\begin{remark}
The non-vanishing of $c_\pi$ in the archimedean place is implicit in \cite[\S2]{LMao5},
but it is reduced to the square-integrable case. In the latter case it can be proved as in \cite[\S6.3]{1203.0039}.
Namely, suppose that $\tilde\pi\in\Irr_{\sqr,\gen,\psi_{\tilde N}}\tilde G$ and let $W\in\WhitG(\tilde\pi)$.
Let $\phi$ be any genuine function in $C_c^\infty(\tilde N\bs\tilde G;\psi_{\tilde N}^{-1})$ such that
$\int_{N'\bs G'}W(\animg{g})\phi(\animg{g})\ dg\ne0$ and let
$f\in C_c^\infty(\tilde G)$ be a genuine function such that $\phi=\int_{N'}f(\animg{n}\cdot)\psi_{\tilde N}(n)\ dn$.
Then $\int_{G'}W(\animg{g})f(\animg{g})\ dg=\int_{N'\bs G'}W(\animg{g})\phi(\animg{g})\ dg\ne0$.
Let $f_1$ be the projection of $f$ into the $\tilde\pi\times\d{\tilde\pi}$-isotypic component
$L^2(\tilde G)_{\tilde\pi\times\d{\tilde\pi}}$ of $L^2(\tilde G)$ (as a representation of $\tilde G\times\tilde G$) and let $f_2=f-f_1$.
Then $\int_{G'}W(\animg{g})f_1(\animg{g})\ dg$ is absolutely convergent and hence so is $\int_{G'}W(\animg{g})f_2(\animg{g})\ dg$. However,
$\int_{G'}W(\animg{g})f_2(\animg{g})\ dg=0$ for otherwise we would get a non-zero pairing between $\tilde\pi$ and the orthogonal complement
of $L^2(\tilde G)_{\tilde\pi\times\d{\tilde\pi}}$ in $L^2(\tilde G)$ which is impossible.
Thus, $\int_{G'}W(\animg{g})f_1(\animg{g})\ dg\ne0$. Since $\int_{G'}W(\animg{g})f_1(\animg{g})\ dg=
\int_{N'\bs G'}W(\animg{g})\int_{N'}f_1(\animg{n}\animg{g})\psi_{\tilde N}(n)\ dn\ dg$ we infer that
$\int_{N'}f'(\animg{n})\psi_{\tilde N}(n)\ dn\ne0$ for some matrix coefficient $f'$ of $\tilde\pi$, as required.
\end{remark}

\begin{remark}
The constant $c_\pi$ depends on the choice of Haar measures on $U$
(for the intertwining operator $M(s)$) and $G'$ (but not on $N'$ and $V_\gamma\bs V$).
Recall that by our conventions, the Haar measures on $U$ and $G$ depend on $\psi$ (which determines a self-dual
Haar measure on $F$).
For $U$ we use the identification $U\cong F^{n(2n+1)}$ via the basis $\one^{\symspace_{2n}}_{i,j}$, $i+j\le 2n+1$.
For $G'$ we take the gauge form given (up to a sign) by the wedge of a Chevalley basis of the Lie algebra.
\end{remark}

The next two results are essentially in \cite{MR1954940} and \cite{MR1722953}.

\begin{proposition}
Let $F$ be a $p$-adic field. Then any $\pi\in\Irr_{\mcusp}\GLnn$ is \good.
\end{proposition}

\begin{proof}
By Lemma \ref{lem: simpleind} $\pi\in\Irr_{\meta}\GLnn$. Moreover, let $\pi'$ be the Langlands quotient of $I(\pi,\frac12)$.
By \cite{MR1671452, MR1954940} $J_V(\pi'\otimes\wevinv)\in\Irr_{\cusp,\gen,\psi_{\tilde N}}\tilde G$
and in particular, $\des(\pi)\in\Irr_{\cusp,\gen,\psi_{\tilde N}}\tilde G$
(and hence $\tilde J(\tilde W,W,\Phi,\cdot)$ is entire).
Similarly for $\tilde\pi=\desinv(\pi)$.
Finally, if $\pi$ is of the form \eqref{eq: pi} then $\gamma(\tilde\pi\otimes\tau_i,s,\psi)$ has a simple pole at $s=0$ for all $i$
and $\gamma(\tilde\pi\otimes\pi,s,\psi)$ has a pole of order $k$ at $s=0$.
It follows from the local functional equations \eqref{eq: local functional equation} that
there exists a non-zero scalar $c$ such that
\[
\tilde{J}(\tilde W,W,\Phi,\frac12)=c\tilde{J}(\tilde W,M(\frac12)W,\Phi,-\frac12)
\]
for any $W\in\Ind(\WhitML(\pi))$, $\Phi\in\swrz(F^n)$ and $\tilde W\in\WhitGd(\tilde\pi)$.
Indeed, both numerator and denominator of the proportionality constant on the right-hand side of
\eqref{eq: local functional equation} have zero of order $k$ at $s=\frac12$.
Thus, $\pi$ is good.
\end{proof}

\begin{proposition} \label{prop: unram}
Suppose that $F$ is $p$-adic with $p\ne2$, $\pi\in\Irr_{\gen}\GLnn$ is unramified, unitarizable and self-dual and $\psi$ is unramified.
Then $\pi\in\Irr_{\gen,\meta}\GLnn$.
Let $\tilde\pi=\desinv(\pi)$. Then $\tilde\pi\in\Irr_{\gen,\psi_{\tilde N}^{-1}}\tilde G$,
$\tilde\pi$ is unramified and the $\psi^{-1}$-lift of $\tilde\pi$ is $\pi$.
Moreover, if $W$ is $K$-invariant and $\tilde W$ is $K'$-invariant then
\[
\stint_{N'}\tilde{J}(\tilde\pi(\animg n)\tilde W,W,1_{\OO^n},\frac12)\psi_{\tilde N}( n)\,dn
=\tilde W(e)\whitform(M(\frac12)W,1_{\OO^n},e,-\frac12).
\]
\end{proposition}

\begin{proof}
Since $\pi$ is generic, self-dual and unramified, there exist unramified characters $\chi_1,\dots,\chi_n$ of $F^*$
such that $\pi=\Ind_{B_\GLnn}^{\GLnn}\chi$ (principal series on $\GLnn$) where $\chi$ is the character
\[
\chi(\diag(t_1,\dots,t_{2n}))=\prod_{i=1}^n\chi_i(t_{2i-1}t_{2i}^{-1}).
\]
By Lemma \ref{lem: simpleind} it follows that $\pi\in\Irr_{\gen,\meta}\GLnn$.
Also, since $\pi$ is irreducible, we have
\begin{equation} \label{eq: nopair}
\chi_i\chi_j^{-1}\ne\abs{\cdot}^{\pm1}\text{ for all }i,j.
\end{equation}
We note that the image of $M(\frac12)$ is irreducible\footnote{This is in fact true for any unitarizable self-dual generic representation $\pi$,
cf.~\cite[p. 904]{MR1983784}. This probably holds without the unitarizability assumption, in which case we can drop this
assumption in the proposition} and it is unramified by \eqref{eq: unramwhit}.
By \cite[Theorem 6.4, part 1]{MR2848523} $\tilde\pi$ is a subquotient of $\Ind_{\tilde B}^{\tilde G}\tilde\chi$
where $\tilde\chi$ is the character of $\tilde T$ given by
\[
\tilde\chi(\levi'(\diag(t_1,\dots,t_n)),1)=\chi_1(t_1)\dots\chi_n(t_n)\gamma_{\psi}(t_1\dots t_n).
\]
We can apply the argument of \cite[\S7]{MR1266251} (cf.~also \cite[Theorem 3.3]{MR1420710})
to show that $\Ind_{\tilde B}^{\tilde G}\tilde\chi$ is irreducible and therefore equals $\tilde\pi$.
(The Langlands classification for $\Mp_n$ was proved in \cite{MR1882038}.)
It is also clear that the $\psi^{-1}$-lift of $\tilde\pi$ is $\pi$.

The last part follows from Lemma \ref{L: whitform} part \ref{part: unramwhitform} and the relations \cite[(2.9)]{LMao5},
\eqref{eq: Junram} and \eqref{eq: unramwhit}.
\end{proof}

\section{Reduction to a local statement} \label{sec: local to global}

In this section we will reduce Conjecture \ref{conj: metplectic global} to a local conjectural identity.

We go back to the global setting.
Let $\K=\prod_vK_v$ be the standard maximal compact subgroup of $G(\A)$.
We will use results from \cite{MR2848523}, where the choice of the character  $\psi_{N_\GLnn}$ differs from ours.
In the following (\S\ref{sec: descent}, \ref{sec: grswhit}), let $\psi_{N_\GLnn}$ be an arbitrary non-degenerate character of $\GLnn(F)\bs \GLnn(\A)$.
The other characters are determined by $\psi_{N_\GLnn}$ as described in \S\ref{sec: characters}.

\subsection{The descent} \label{sec: descent}
Consider the set $\Mcusp\GLnn$ of automorphic representations $\pi$ of $\GLnn(\A)$ which are realized on Eisenstein series
induced from $\pi_1\otimes\dots\otimes\pi_k$ where $\pi_i\in\Cusp_{\meta}\GL_{2n_i}$, $i=1,\dots,k$ are distinct and $n=n_1+\dots+n_k$.
The representation $\pi$ is irreducible: it is equivalent to the parabolic induction $\pi_1\times\dots\times\pi_k$.
Moreover, $\pi$ determines $\pi_1,\dots,\pi_k$ uniquely up to permutation \cite{MR618323, MR623137}.

We view $\pi$ as a representation of $\Levi(\A)$ via $\levi$.
Note that the space of $\pi$ is invariant under conjugation by $w_U$.
Let $\AF(\pi)$ be the space of functions $\varphi:\Levi(F)U(\A)\bs G(\A)\rightarrow\C$
such that $m\mapsto\modulus_P(m)^{-\frac12}\varphi(mg)$, $m\in\Levi(\A)$ belongs to the space of $\pi$ for all $g\in G(\A)$.
As in the local case, let $\varphi_s(g)=\wgt{g}^s\varphi(g)$.
We define Eisenstein series
\[
\eisen(\varphi,s)=\sum_{\gamma\in P(F)\bs G(F)}\varphi_s(\gamma g), \ \ \varphi\in\AF(\pi)
\]
and intertwining operator $M(s):\AF(\pi)\rightarrow\AF(\pi)$ given by
\[
M(s)\varphi(g)=\wgt{g}^s\int_{U(\A)}\varphi_s(w_Uug)\ du.
\]
Both converge absolutely for $\Re s\gg0$ and extend meromorphically to $\C$.
If follows from \cite[Theorem 2.1]{MR2848523} that both $\eisen(\varphi,s)$ and $M(s)$ have a pole of order $k$ at $s=\frac12$.
Let
\[
\reseisen\varphi=\lim_{s\rightarrow\frac12}(s-\frac12)^k\eisen(\varphi,s)
\]
and
\[
\resM=\lim_{s\rightarrow\frac12}(s-\frac12)^kM(s)
\]
be the leading terms in the Laurent expansion at $s=\frac12$.
The constant term
\[
\reseisen^U\varphi(g)=\int_{U(F)\bs U(\A)}\reseisen\varphi(ug)\ du
\]
of $\reseisen\varphi$ along $U$ is given by
\[
\reseisen^U\varphi(g)=(\resM\varphi(g))_{-\frac12}.
\]

Recall the Fourier-Jacobi coefficient defined in \eqref{def: FJ}.
By definition, the \emph{descent} of $\pi$ (with respect to $\psi_{N_\GLnn}$) is the space
$\tilde\pi=\desc_{\psi_{N_\GLnn}}(\pi)$ generated by $\FJ(\reseisen\varphi,\Phi)$, $\varphi\in\AF(\pi)$, $\Phi\in\swrz(\A^n)$.
The following result summarizes the work of Ginzburg--Rallis--Soudry, Ginzburg--Jiang--Soudry, and
Cogdell--Kim--Piatetski-Shapiro--Shahidi.

As in \cite{LMao5} let $\Cusp_{\psi_{\tilde N}}\tilde G$ be the set of irreducible constituents of the
space of cusp forms on $\tilde G(\A)$ which are orthogonal to all automorphic forms with vanishing $\psi_{\tilde N}$-Whittaker functions.
\begin{theorem}
\begin{enumerate}
\item $\tilde\pi\in\Cusp_{\psi_{\tilde N}}\tilde G$.
\item The $\psiweil$-lift of $\tilde\pi$ is $\pi$.
\item The descent map defines a bijection between $\Mcusp\GLnn$ and
the set of $\sigma$'s in $\Cusp_{\psi_{\tilde N}}\tilde G$ with trivial $\psiweil$-theta lift to $SO(2n-1)$.
\end{enumerate}
\end{theorem}

The first two parts, except for the irreducibility, are contained in \cite[Theorem 3.1]{MR2848523}.
The irreducibility part is \cite[Theorem 2.3]{MR3009748}.
Finally, the last part is \cite[Theorem 11.2]{MR2848523} (which uses results on the theta correspondence
by Jiang-Soudry \cite{MR1983781, MR2330445}) together with the results
of Cogdell--Kim--Piatetski-Shapiro--Shahidi \cite{MR2075885} in the case of the odd orthogonal groups.

For the rest of the section (up to \S\ref{sec: proof}) we will show the following:
\begin{theorem} \label{thm: local to global}
Let $\pi\in\Mcusp\GLnn$ and $k$ as above. Then for all $v$ $\pi_v$ is \good.
Moreover, let $S$ be a finite set of places including all the archimedean and even places
such that $\pi$ and $\psi$ (and hence also $\tilde\pi$) are unramified outside $S$.
Then for any $\tilde\varphi\in\tilde\pi$ and $\tilde\varphi^\vee\in\tilde\pi^\vee$ which are fixed under $K'_v$ for all $v\notin S$ we have
\label{eq: mainidtmodinfty}
\begin{multline} \label{eq: modglobalidentity}
\tilde\whit^{\psi_{\tilde N}}(\tilde\varphi)\tilde\whit^{\psi_{\tilde N}^{-1}}(\tilde\varphi^\vee)=
2^{-k}\Delta_{\Sp_n}^S(1)(\prod_{v\in S}c_{\pi_v}^{-1})\frac{L^S(\frac12,\pi)}{L^S(1,\pi,\sym^2)}\\
\stint_{N'(F_S)}(\tilde\pi(\animg{n})\tilde\varphi,\tilde\varphi^\vee)_{\Sp_n(F)\bs\Sp_n(\A)}\psi_{\tilde N}(n)^{-1}\ dn.
\end{multline}
\end{theorem}

Note that $\vol(N'(\OO_S)\bs N'(F_S))=1$ by our choice of measures.

\subsection{}\label{sec: grswhit}
Recall that \cite[Theorem 9.7, part (1)]{MR2848523} states a formula
for the Whittaker--Fourier coefficient of $\FJ(\reseisen\varphi,\Phi)$.
We will rewrite it in a different way as follows.

\begin{theorem}(reformulation of \cite[Theorem 9.7, part (1)]{MR2848523}) \label{thm: FCdescent}
We have
\[
\tilde\whit^{\psi_{\tilde N}}(\FJ(\reseisen\varphi,\Phi),\animg g)=
\int_{V_\gamma(\A)\bs V(\A)}\whit^{\psi_N}(\reseisen\varphi,\gamma v \toG(g))
\wevinv(v\animg g)\Phi(\xi_n)\,dv, \ \ g\in G'(\A)
\]
where $\xi_n=(0,\dots,0,1)\in F^n$ and the integral is absolutely convergent.
\end{theorem}

\begin{remark}\label{rem: diffgrs}
Before proving the theorem, we first observe that if it holds for $\psi_{N_\GLnn}=\psi_{N_\GLnn}^1$
(and the compatible set of characters determined by $\psi_{N_\GLnn}$ as in \S\ref{sec: characters})
then it also holds for $\psi_{N_\GLnn}=\psi_{N_\GLnn}^2$ as long as $\psiweil^1=\psiweil^2$.
Indeed, in this case we can write $\psi_{N_\GLnn}^1(n)=\psi_{N_\GLnn}^2(\spclt_1 n\spclt_1^{-1})$ where $\spclt_1=\diag[\spclt_2,\spclt_3^*]$
is such that $\spclt_2,\spclt_3\in T'_{\GLn}$ where their last diagonal entry is $1$.
The straightforward relations
\begin{gather*}
\FJc_{\psi_{N_\GLnn}^1}(\varphi,\Phi)=\FJc_{\psi_{N_\GLnn}^2}(I(\frac12, \toLevi(\spclt_3))\varphi,\Phi),\\
\tilde\whit^{\psi_{\tilde N}^1}(\tilde\varphi,\tilde g)=\tilde\whit^{\psi_{\tilde N}^2}(\tilde\varphi,\widetilde {\levi'(\spclt_2)}\tilde g),\\
\whit^{\psi_{N_\Levi}^1}(\varphi, g)=\whit^{\psi_{N_\Levi}^2}(\varphi, \levi(\spclt_1)g)
\end{gather*}
imply out claim.
The same observation holds for Conjecture~\ref{conj: metplectic global}, Theorem~\ref{thm: local to global} as well as for Theorem~\ref{thm: RS} below.
\end{remark}

\begin{proof}
Note that the integrand on the right-hand side is indeed left $V_\gamma(\A)$-invariant exactly as in the proof of Lemma \ref{L: whitform}.
By the remark above we only need to prove the theorem for a specific choice of $\psi_{N_\GLnn}$.
We use the definition of \cite{MR2848523}:
\[\psi_{N_\GLnn}(u)=\psi(u_{1,2}+\ldots+u_{n-1,n}+2u_{n,n+1}-u_{n+1,n+2}-\ldots-u_{2n-1,2n}),\ \ \ u\in N_\GLnn.\]
Thus $\psiweil(x)=\psi(2x)$.
It is enough to prove the required identity for $g=e$.
The expression for $\tilde\whit^{\psi_{\tilde N}}(\FJ(\reseisen\varphi,\Phi),e)$
in \cite[Theorem 9.7, part (1)]{MR2848523} is\footnote{Note the following typo in [loc.~cit.]: the term $\omega_{\psi^{-1},\gamma^{-1}}(h)$
should not appear in the formula.}
\[
\int_{\A^n}\int_{Y(\A)}\big(\int_{X(\A)}\whit^{\psiold_{N}}(\reseisen\varphi,\toUbar(x)\alpha\epsilon\kappa ya^+)\Phi(a)\ dx\big)\ dy\ da
\]
where we use the following notation
\begin{itemize}
\item $\alpha$ is the Weyl element such that
$\alpha_{i,2i-1}=1$, $i=1,\dots,2n$, $\alpha_{2n+i,2i}=-1$, $i=1,\dots,n$,
$\alpha_{2n+i,2i}=1$, $i=n+1,\dots,2n$,
\item $\kappa=\levi(\kappa')$ where $\kappa'$ is the Weyl element of $\GLnn$ such that
$\kappa'_{2i,i}=\kappa'_{2i-1,n+i}=1$, $i=1,\dots,n$,
\item $\epsilon=\diag(A,\dots,A,A^*,\dots,A^*)$ where $A=\sm 11{-1}1$,
\item $X$ is the subspace of $\symspace_{2n}$ consisting of the strictly upper triangular matrices.
\item $Y$ is the subgroup of $N_\Levi$ consisting of the matrices of the form
$\levi(\toU_\GLnn(y))$ where $y$ is lower triangular with last row $0$.
\item $a^+=\levi(\toU_\GLnn(y))$ where the last row of $y$ is $a$ and all other rows
are zero.
\item $\psiold_N$ is a degenerate character on $N$ that is trivial on $U$ and for $u\in N_\GLnn$
\[
\psiold_N(\levi(u)):=\psi(u_{1,2}+\dots+u_{n,n+1}-u_{n+1,n+2}-\dots-u_{2n-1,2n}).
\]
\end{itemize}
Combining $y$ and $a^+$, by \eqref{eq: weilH1} and \eqref{eq: weilext} we can rewrite the above expression as
\[
\int_{Y'(\A)}\big(\int_{X(\A)}\whit^{\psiold_{N}}(\reseisen\varphi,\toUbar(x)\alpha\epsilon\kappa y)\wevinv(y)\Phi(0)\ dx\big)\ dy
\]
where now
\[
Y'=\{\levi(\toU_\GLnn(y)):y\text{ lower triangular}\}.
\]

Writing
\[
A=\sm 1101\sm 2001\sm 10{-1}1
\]
we get
$\epsilon=\epsilon_U\epsilon_T\epsilon_{\bar U}$
where
\[
\epsilon_U=\diag(\sm 1101,\dots,\sm 1101,\sm 1{-1}01,\dots,\sm 1{-1}01),
\]
\[
\epsilon_{\bar U}=\diag(\sm 10{-1}1,\dots,\sm 10{-1}1,\sm 1011,\dots,\sm 1011),
\]
and
\[
\epsilon_T=\diag(2,1,\dots,2,1,1,\frac12,\dots,1,\frac12).
\]

Note that
$\alpha\epsilon_U\alpha^{-1}=\sm{I_{2n}}{-I_{2n}}0{I_{2n}}.$
For any $x\in X(\A)$ we can write
\[
\sm{I_{2n}}0x{I_{2n}}\sm{I_{2n}}{-I_{2n}}0{I_{2n}}=\sm{(I_{2n}-x)^{-1}}{-I_{2n}}{}{I_{2n}-x}\sm{I_{2n}}0{x'}{I_{2n}}
\]
where $x'=(I_{2n}-x)^{-1}x=x+x^2+\dots+x^{2n-1}$. (Note that $(I_{2n}-x)^{-1}=(I_{2n}-x)^*$.)
After a change of variable $x'\mapsto x$ we get
\[
\int_{Y'(\A)}\big(\int_{X(\A)}\psi(x_{n,n+1})\whit^{\psiold_{N}}(\reseisen\varphi,
\toUbar(x)\alpha\epsilon_T\epsilon_{\bar U}\kappa y)\wevinv(y)\Phi(0)\ dx\big)\ dy
\]
since $\psiold_{N}(\levi(I_{2n}-x)^{-1})=\psi(x_{n,n+1})=\psi(x'_{n,n+1})$.

Also, $\alpha\epsilon_T\alpha^{-1}=\diag(2I_n,I_n,I_n,\frac12 I_n)$.
This element conjugates $\psiold_{N}$ to $\psi_{N}$ and stabilizes $\psi(x_{n,n+1})$,
so by conjugating we get
\[
\int_{Y'(\A)}\big(\int_{X(\A)}\psi(x_{n,n+1})\whit^{\psi_{N}}(\reseisen\varphi,
\toUbar(x)\alpha\epsilon_{\bar U}\kappa y)\wevinv(y)\Phi(0)\ dx\big)\ dy.
\]
Note that $\alpha\epsilon_{\bar U}\kappa=\gamma\levi(\toU_\GLnn(-I_n))$.
Changing variable  $y\mapsto \levi(\toU_\GLnn(I_n))y$, we get by \eqref{eq: weilH1} and \eqref{eq: weilext}
\[
\int_{Y'(\A)}\big(\int_{X(\A)}\psi(x_{n,n+1})\whit^{\psi_{N}}(\reseisen\varphi,\toUbar(x)\gamma y)\wevinv(y)\Phi(\xi_n)\ dx\big)\ dy.
\]
Note that
\[
\gamma^{-1}\toUbar(X)\gamma=\{\left(\begin{smallmatrix}I_n&x_1&&x_2\\&I_n&&\\&&I_n&-\startran{x_1}\\&&&I_n\end{smallmatrix}\right):
x_1\text{ strictly upper triangular}, x_2\in\symspace_n\},
\]
and thus $(\gamma^{-1}\toUbar(X)\gamma Y')\rtimes V_\gamma=V$.
In conclusion we obtain
\[
\int_{V_\gamma(\A)\bs V(\A)}\whit^{\psi_{N}}(\reseisen\varphi,\gamma x)\wevinv(x)\Phi(\xi_n)\ dx
\]
provided it converges, since by \eqref{eq: weilH3} and \eqref{eq: weilext} $\wevinv(\gamma^{-1}\toUbar(x)\gamma)$ acts by the scalar $\psi(x_{n,n+1})$ for $x\in X(\A)$.

Finally, the absolute convergence follows from Lemma \ref{L: whitform}.
\end{proof}

\subsection{}

We now go back to our choice of $\psi_{N_\GLnn}$ specified in \S\ref{sec: characters}.

We can rewrite Theorem \ref{thm: FCdescent} in terms of the local transforms
$\whitform$ defined in \S\ref{sec: local conjecture}.

Note that
\[
\whit^{\psi_N}(\reseisen\varphi)=\whit^{\psi_{N_\Levi}}(\reseisen^U\varphi)=
\whit^{\psi_{N_\Levi}}((\resM(\varphi))_{-\frac12}).
\]

Define for $\varphi\in\AF(\pi)$ a genuine function on $\tilde G(\A)$
\[
\whitform(\varphi,\Phi,\animg g,s):=\int_{V_\gamma(\A)\bs V(\A)}
\whit^{\psi_{N_\Levi}}(\varphi_s,\gamma v\toG(g))\wevinv(v\animg g)\Phi(\xi_n)\,dv,\,\,g\in G'(\A).
\]
Then we can write Theorem \ref{thm: FCdescent} in the form
\begin{equation} \label{eq: whitdescent}
\tilde\whit^{\psi_{\tilde N}}(\FJ(\reseisen\varphi,\Phi),\tilde g)=\whitform(\resM(\varphi),\Phi,\tilde g,-\frac12).
\end{equation}

\subsection{Factorization}

We can identify $\AF(\pi)$ with $\Ind_{P(\A)}^{G(\A)}\pi=\otimes_v\Ind_{P(F_v)}^{G(F_v)}\pi_v$.

Suppose that $\varphi\in\AF(\pi)$ is a factorizable vector and
$\whit^{\psi_{N_\Levi}}(\varphi,\cdot)=\prod_vW_v$ where $W_v\in\Ind(\WhitML(\pi_v))$.
Then as meromorphic functions in $s\in\C$ we have for $S$ large enough
\begin{equation} \label{eq: globtoloc1}
\whit^{\psi_{N_\Levi}}(M(s)\varphi)=m^S(\pi,s)(\prod_{v\in S}M_v(s)W_v)\prod_{v\notin S}W_v
\end{equation}
where
\[
m^S(\pi,s)=\frac{L^S(s,\pi)}{L^S(s+1,\pi)}\frac{L^S(2s,\pi,\wedge^2)}{L^S(2s+1,\pi,\wedge^2)}
\]
and $M_v(s):\Ind(\WhitML(\pi_v),s)\rightarrow\Ind(\WhitML(\d\pi_v),-s)$ are the local intertwining operators defined in \eqref{eq: defM}.
(Recall that $\pi=\d\pi$ in out case.)
Indeed, for $\Re s\gg0$ we have
\begin{multline*}
\whit^{\psi_{N_\Levi}}(M(s)\varphi,g)=\int_{N_\Levi(F)\bs N_\Levi(\A)}M(s)\varphi(ng)\psi_{N_\Levi}^{-1}(n)\,dn
\\=\int_{N_\Levi(F)\bs N_\Levi(\A)}\int_{U(\A)} \varphi(\levi(\spclt)w_U ung)\psi_{N_\Levi}^{-1}(n)\,du\,dn
\\=\int_{U(\A)} \whit^{\psi_{N}}(\varphi, \levi(\spclt)w_U ug)\,du
\end{multline*}
and the relation \eqref{eq: globtoloc1} follows from \eqref{eq: unramwhit}.
It follows that when $S$ is sufficiently large,
\[
\whit^{\psi_{N_\Levi}}(\resM\varphi)=\resm^S(\pi)\prod_{v\in S}M_v(\frac12)W_v\prod_{v\notin S}W_v
\]
where
\[
\resm^S(\pi)=\lim_{s\rightarrow\frac12}(s-\frac12)^k m^S(\pi,s)=
2^{-k}\frac{L^S(\frac12,\pi)}{L^S(\frac32,\pi)}
\frac{\lim_{s\rightarrow1}(s-1)^kL^S(s,\pi,\wedge^2)}{L^S(2,\pi,\wedge^2)}.
\]
We conclude from \eqref{eq: whitdescent} and Lemma \ref{L: whitform} that
for any factorizable $\varphi\in\AF(\pi)$ we have
\begin{equation} \label{eq: Fourierofdescent}
\tilde\whit^{\psi_{\tilde N}}(\FJ(\reseisen\varphi,\Phi),\tilde g)=\resm^S(\pi)
\prod_{v\in S}\whitform_v(M_v(\frac12)W_v,\Phi_v,\tilde g_v,-\frac12)
\end{equation}
for any $\Phi=\otimes_v\Phi_v\in\swrz(\A^n)$ where
$\whit^{\psi_{N_\Levi}}(\varphi,\cdot)=\prod_vW_v$.
Of course, $S$ is large enough so that $\Phi_v$ is the standard function for $v\notin S$.

At this stage we can conclude that $\des(\pi_v)$ is irreducible for all $v$.
This follows from \eqref{eq: Fourierofdescent} and the irreducibility of the global descent.

\subsection{Rankin--Selberg integrals}

Recall that we use the Tamagawa measure on $G'(\A)$ and $G'(F_v)$
with respect to the standard gauge form on $G'$ and the character $\psi$.
We have $\vol(K'_v)=\Delta_{\Sp_n,v}(1)^{-1}$ for all finite places.

For any genuine functions $\phi_1$, $\phi_2$ on $G'(F)\bs\tilde G(\A)$ we write
\[
\sprod{\phi_1}{\phi_2}_{G'}=\int_{G'(F)\bs G'(\A)}\phi_1(g)\phi_2(g)\ dg
\]
provided that the integral converges.
We summarize the main properties of the Ginzburg--Rallis--Soudry zeta integrals defining the Rankin--Selberg $L$-functions
for $\tilde G'\times\GLnn$.

\begin{theorem} (\cite[Theorem 10.4, (10.6), (10.63)]{MR2848523})  \label{thm: RS}
Let $\tilde\sigma\in\Cusp_{\psi_{\tilde N}^{-1}}\tilde G$
and $\pi$ an automorphic representation of $\GLnn(\A)$ considered as a representation of $\Levi(\A)$.
For any $\tilde\varphi\in\tilde\sigma$, $\varphi\in\AF(\pi)$ and $\Phi\in\swrz(\A^n)$
we have
\[
\sprod{\tilde\varphi}{\FJ(\eisen(\varphi,s),\Phi)}_{G'}=
\int_{N'(\A)\bs G'(\A)}\tilde\whit^{\psi_{\tilde N}^{-1}}(\tilde\varphi,\tilde g)
\whitform(\varphi,\Phi,\tilde g,s)\ dg
\]
for $\Re s\gg0$. The left-hand side is a meromorphic function in $s$ which is holomorphic whenever
$\eisen(\varphi,s)$ is.
Furthermore, if $\Phi=\otimes\Phi_v$, $\whit^{\psi_{N_\Levi}}(\varphi,g)=\prod_vW_v(g_v)$ and
$\tilde\whit^{\psi_{\tilde N}^{-1}}(\tilde\varphi)=\prod_v\tilde W_v$ then
for any sufficiently large finite set of places $S$ we have
\begin{equation}\label{eq: globalinner}
\sprod{\tilde\varphi}{\FJ(\eisen(\varphi,s),\Phi)}_{G'}=
\Delta_{\Sp_n}^S(1)^{-1}
\frac{L^S_\psi(s+\frac12,\tilde\sigma\otimes\pi)}
{L^S(s+1,\pi)L^S(2s+1,\pi,\wedge^2)}
\prod_{v\in S}\tilde{J}_v(\tilde W_v,W_v,\Phi_v,s).
\end{equation}
\end{theorem}

\begin{remark}
The factor $\Delta_{\Sp_n}^S(1)^{-1}$ shows up because of our choice of global and local measures.
Note that in \cite[Theorem 10.4, (10.6)]{MR2848523} the element
$\left(\begin{smallmatrix}&I_n&&\\&&&-I_n\\I_n&&&\\&&I_n&\end{smallmatrix}\right)$
was used instead of $\gamma$. Also note the following typos in [loc.~cit.]:
on p. 294 the definition of $\eta=\eta_0$ should be changed to $\diag(I_l,\sm{}{I_{m-l-1}}1{}^\vee
\left(\begin{smallmatrix}&&-1\\&I_{2(m-l-1)}&\\1&&\end{smallmatrix}\right),I_l)$
so that the formula on the last line of p. 296 holds with $-x'$ instead of $x'$,
which is necessary for the ensuing discussion on p. 297.
Also, in the definition of the character $\psi_{N_G}$ on p. 298 the $z_{i,i+1}$'s should be replaced
by their negatives in order for the last equation on that page to hold.
This is compatible with the fact that the character in [ibid., (10.7)]
differs from $\psi_{N_\Levi}$ used in $\whitform$.
\end{remark}

In particular when $\pi\in\Mcusp\GLnn$ and $\tilde\varphi\in\tilde\sigma=\desinv(\pi)$ we have $L^S_\psi(s,\tilde\sigma\otimes\pi)=L^S(s,\pi\otimes\pi)$.
Thus, the right-hand side of \eqref{eq: globalinner} has a pole of order at least $k$ at $s=\frac12$ for suitable $\varphi$, $\tilde\varphi$ and $\Phi$,
since $\tilde{J}_v(\tilde W_v,W_v,\Phi_v,s)$ is non-vanishing at $s=\frac12$ for suitable $W_v$, $\tilde W_v$ and $\Phi_v$.
On the other hand, the left-hand side of \eqref{eq: globalinner} has a pole of order $\le k$ because this is true for $\eisen(\varphi,s)$ and $\tilde\varphi$ is rapidly decreasing.
Multiplying \eqref{eq: globalinner} by $(s-\frac12)^k$ and taking the limit as $s\rightarrow\frac12$
we conclude that $\tilde{J}_v(\tilde W_v,W_v,\Phi_v,s)$ is holomorphic at $s=\frac12$ for all $v$ (a fact which we already noted before) and
\begin{equation} \label{eq: innerprod}
\sprod{\tilde\varphi}{\FJ(\reseisen\varphi,\Phi)}_{G'}=
\Delta_{\Sp_n}^S(1)^{-1}
\frac{\lim_{s\rightarrow1}(s-1)^kL^S(s,\pi\otimes\pi)}
{L^S(\frac32,\pi)L^S(2,\pi,\wedge^2)}
\prod_{v\in S}\tilde{J}_v(\tilde W_v,W_v,\Phi_v,\frac12).
\end{equation}

\subsection{Proof of Theorem \ref{thm: local to global}} \label{sec: proof}

We already showed that for all $v$ $\desinv(\pi_v)$ (and similarly $\des(\pi_v)$) is irreducible and that $\tilde{J}_v(\tilde W_v,W_v,\Phi_v,s)$
is holomorphic at $s=\frac12$. Fix a place $v_0$.
To show that \eqref{eq: factorsthru} holds for $\pi_{v_0}$
suppose that $\Phi=\otimes\Phi_v$, $\whit^{\psi_{N_\Levi}}(\varphi,g)=\prod_vW_v(g_v)$ and
$\tilde\whit^{\psi_{\tilde N}^{-1}}(\tilde\varphi)=\prod_v\tilde W_v$.
Assume that $W_{v_0}$ and $\Phi_{v_0}$ are such that $\whitform(M_{v_0}(\frac12)W_{v_0},\Phi_{v_0},\cdot,-\frac12)\equiv 0$.
Then by \eqref{eq: Fourierofdescent}  $\tilde\whit^{\psi_{\tilde N}}(\FJ(\reseisen\varphi,\Phi),\cdot)\equiv 0$ and therefore $\FJ(\reseisen\varphi,\Phi)\equiv0$
by the irreducibility and genericity of the descent. (We thank Atsushi Ichino for this observation.)
By \eqref{eq: innerprod} we conclude that $\tilde{J}_{v_0}(\tilde W_{v_0},W_{v_0},\Phi_{v_0},\frac12)=0$.
Thus, $\pi_{v_0}$ is good.

By Remark~\ref{rem: diffgrs}, it remains to show \eqref{eq: modglobalidentity} with our specific choice of compatible set of characters.
Suppose that $\pi\in\Mcusp\GLnn$.
Since the local components of $\pi_1,\dots,\pi_k$ are of metaplectic type, it follows
from Lemma \ref{lem: simpleind} that the same is true for $\pi$.
We know that both sides of \eqref{eq: globalidentity} are proportional. Let $c_\pi$ be the constant of proportionality.
Combining \eqref{eq: innerprod} with \eqref{eq: Fourierofdescent} we conclude \eqref{eq: main} holds for $\pi_v$ for all $v$ with
$c_{\pi_v}=1$ for almost all $v$ (by Proposition \ref{prop: unram}), and that $c_\pi=\prod_vc_{\pi_v}$.

\subsection{Local conjecture} \label{sec: localconj}
Let $F$ be a local field of characteristic $0$.
Motivated by Theorem  \ref{thm: local to global} we make the following:
\begin{conjecture} \label{conj: goodeps}
For any \good\ unitarizable $\pi\in\Irr_{\gen,\meta}\GLnn$ we have $c_\pi=\epsilon(\frac12,\pi,\psi)$.
\end{conjecture}
(Recall that $\eps_\pi=\epsilon(\frac12,\pi,\psi)$ in the $p$-adic case.)
By Theorem \ref{thm: local to global}, Conjecture \ref{conj: goodeps} implies Conjecture \ref{conj: metplectic global}
since $\prod_v\epsilon(\frac12,\pi,\psi_v)=1$ as $L(\frac12,\pi)\ne0$.

Recall that by Proposition \ref{prop: unram}, Conjecture \ref{conj: goodeps} holds in the unramified case.

In fact, it is natural to expect the following stronger version of Conjecture \ref{conj: goodeps}.

\begin{conjecture} \label{conj: later}
Any $\pi\in\Irr_{\gen,\meta}\GLnn$ is \good\ and satisfies $c_\pi=\epsilon(\frac12,\pi,\psi)$.
\end{conjecture}

\section{The case $n=1$ -- formal computation} \label{sec: n=1}

We will substantiate Conjecture \ref{conj: later} (or more precisely, the identity $c_\pi=\eps_\pi$) by considering the case $n=1$
which is already non-trivial and contains many of the ingredients for the general case.
We will confine ourselves in this section to a {\bf heuristic argument},
treating all integrals as if they were absolutely convergent.

Throughout let $F$ be a local field.

We fix an irreducible generic representation $\pi$ of $\GLnn=\GL_2$ with a
trivial central character. For convenience we will also consider $\pi$ as a representation of $\Levi$ through $\levi$.
Recall that $\psi_{N_\GLnn}(\sm{1}{x}{}{1})=\psi(x)$ and $ \psi_{\tilde N}(\sm{1}{x}{}{1})=\psi(-\frac12 x)$.

\subsection{}
We first  describe $\whitform(W,\Phi,\tilde g,s)$ for $g=w'_{U'}mu$ in the big cell $U'w'_{U'}P'$, where $m\in \Levi'$ and $u\in U'$.
From \eqref{eq: defwhitform},
\[
\whitform(W,\Phi,\widetilde{{w'_{U'}}}\animg m\animg u,s)=
\int_{V_\gamma\bs V}W_s(\gamma v \toG(w'_{U'}mu))\wevinv(v\widetilde{w'_{U'}}\animg m\animg u)\Phi(\xi_n)\ dx\ dy.
\]
Recall $\wevinv$ was defined in \eqref{eq: weilext}, $\gamma=w_U\toG(w'_{U'})^{-1}$ and $V_\gamma=\toG(w'_{U'})(V\cap N_\Levi)\toG(w'_{U'})^{-1}$.
It is convenient to make a change of variable $v\mapsto \toG(w'_{U'})v\toG(w'_{U'})^{-1}$. Note that $\toG(w'_{U'})$ normalizes $V$.
By \eqref{eq: weilext2} the integral becomes:
\[
\int_{(V\cap N_\Levi)\bs V}W_s(w_U v \toG(mu))\wevinv(\widetilde{w'_{U'}} v\animg m\animg u)\Phi(\xi_n)\ dv.
\]

When $n=1$, the group of $\{v(x,y)=\toU(\sm{x}{y}{}{x})\}$ is a section of $(V\cap N_\Levi)\bs V$. Explicitly we have:
\begin{multline}\label{eq: n=1whitform}
\whitform(W,\Phi,\widetilde{\sm{}1{-1}{}}\widetilde{\sm t{tz}{}{t^{-1}}},s)=
\\
\iint_{F^2}
W_s(\left(\begin{smallmatrix}&&1&\\&&&1\\{-1}&&&\\&-1&&\end{smallmatrix}\right)
\left(\begin{smallmatrix}1&&x&y\\&1&&x\\&&1&\\&&&1\end{smallmatrix}\right)
\left(\begin{smallmatrix}1&&&\\&t&tz&\\&&t^{-1}&\\&&&1\end{smallmatrix}\right))
\wevinv(\widetilde{\sm{}1{-1}{}} \left(\begin{smallmatrix}1&&x&y\\&1&&x\\&&1&\\&&&1\end{smallmatrix}\right)
\widetilde{\sm t{tz}{}{t^{-1}}})\Phi(1)\ dv.
\end{multline}

For $\Phi\in\swrz(F)$ and $f\in C^\infty(\Sp_2)$ define
\[
\Phi*f(g)=\int_F f(g\left(\begin{smallmatrix}1&r&&\\&1&&\\&&1&{-r}\\&&&1\end{smallmatrix}\right))\Phi(r)\ dr.
\]

\begin{claim}
\begin{multline} \label{eq: whitn1}
\whitform(W,\Phi,\widetilde{\sm{}1{-1}{}}\widetilde{\sm{t}{tz}{}{t^{-1}}},s)
=\\\abs{t}^{\frac12}\nu_\psi
\iint_{F^2}(\Phi*(W_s))(\left(\begin{smallmatrix}t^{-1}&&&\\&1&&\\&&1&\\&&&t\end{smallmatrix}\right)
\left(\begin{smallmatrix}1&&&\\&1&&\\x&y&1&\\-z&x&&1\end{smallmatrix}\right)
\left(\begin{smallmatrix}&&1&\\&&&1\\{-1}&&&\\&-1&&\end{smallmatrix}\right))\psi(\frac12 y)\,dx\,dy
\end{multline}
where $\nu_\psi$ is a certain root of unit determined by $\psi$ satisfying $\nu_\psi\nu_{\psi^{-1}}=1$.
\end{claim}
\begin{proof}[``Proof"]
Using \eqref{eq: weil2} we get the expression \eqref{eq: n=1whitform} is
\[\nu_\psi'
\iiint_{F^3}
W_s(\left(\begin{smallmatrix}&&1&\\&&&1\\{-1}&&&\\&-1&&\end{smallmatrix}\right)
\left(\begin{smallmatrix}1&&x&y\\&1&&x\\&&1&\\&&&1\end{smallmatrix}\right)
\left(\begin{smallmatrix}1&&&\\&t&tz&\\&&t^{-1}&\\&&&1\end{smallmatrix}\right))
\wevinv(\left(\begin{smallmatrix}1&&x&y\\&1&&x\\&&1&\\&&&1\end{smallmatrix}\right)
\widetilde{\sm t{tz}{}{t^{-1}}})\Phi(Y)\psi(-Y)\,dY  \ dx\ dy.\]
Here $\nu_\psi'$ is a certain root of unit (determined by $\psi$) satisfying $\nu_\psi'\nu_{\psi^{-1}}'=1$.

By the equivariance of $W$ and  \eqref{eq: weilH1}, this is
\begin{multline*}
\nu_\psi'
\iiint_{F^3}
W_s(\left(\begin{smallmatrix}&&1&\\&&&1\\{-1}&&&\\&-1&&\end{smallmatrix}\right)
\left(\begin{smallmatrix}1&Y&&\\&1&&\\&&1&-Y\\&&&1\end{smallmatrix}\right)
\left(\begin{smallmatrix}1&&x&y\\&1&&x\\&&1&\\&&&1\end{smallmatrix}\right)
\left(\begin{smallmatrix}1&&&\\&t&tz&\\&&t^{-1}&\\&&&1\end{smallmatrix}\right))
\\
\wevinv(\left(\begin{smallmatrix}1&Y&&\\&1&&\\&&1&-Y\\&&&1\end{smallmatrix}\right)
\left(\begin{smallmatrix}1&&x&y\\&1&&x\\&&1&\\&&&1\end{smallmatrix}\right)
\widetilde{\sm t{tz}{}{t^{-1}}})\Phi(0)\,dY  \ dx\ dy.
\end{multline*}
Clearly the triple integral can be considered as the integration over Heisenberg group $\{v(x,y,r)=\toU(\sm{x}{y}{}{x})\levi(\sm{1}{r}{}{1})\}$:
\[
\nu_\psi'\iiint_{F^3}
W_s(\left(\begin{smallmatrix}&&1&\\&&&1\\{-1}&&&\\&-1&&\end{smallmatrix}\right)
v(x,y,r)\left(\begin{smallmatrix}1&&&\\&t&tz&\\&&t^{-1}&\\&&&1\end{smallmatrix}\right))
\wevinv(v(x,y,r)\widetilde{\sm t{tz}{}{t^{-1}}})\Phi(0)\,dr\ dx\ dy.\]
Make the change of variable $v(x,y,r)\mapsto \toG(\sm t{tz}{}{t^{-1}})v(x,y,r) \toG(\sm t{tz}{}{t^{-1}})^{-1}$; from \eqref{eq: weilext2} the integral becomes:
\[
\nu_\psi'\iiint_{F^3}
W_s(\left(\begin{smallmatrix}&&1&\\&&&1\\{-1}&&&\\&-1&&\end{smallmatrix}\right)
\left(\begin{smallmatrix}1&&&\\&t&tz&\\&&t^{-1}&\\&&&1\end{smallmatrix}\right)v(x,y,r))
\wevinv(\widetilde{\sm t{tz}{}{t^{-1}}}v(x,y,r))\Phi(0)\,dr\ dx\ dy.\]
From  \eqref{eq: weil1} and \eqref{eq: weil3}, the above becomes (for some root of unity $\nu_\psi$ satisfying $\nu_\psi\nu_{\psi^{-1}}=1$):
\[
\abs{t}^{\frac12}\nu_\psi\iiint_{F^3}
W_s(\left(\begin{smallmatrix}&&1&\\&&&1\\{-1}&&&\\&-1&&\end{smallmatrix}\right)
\left(\begin{smallmatrix}1&&&\\&t&tz&\\&&t^{-1}&\\&&&1\end{smallmatrix}\right)v(x,y,r))
\wevinv(v(x,y,r))\Phi(0)\,dr\ dx\ dy.\]
From \eqref{eq: weilH1}, \eqref{eq: weilH2} and
\eqref{eq: weilH3},
we get $\wevinv(v(x,y,r))\Phi(0)=\psi(-\frac12 y)\Phi(r)$.
Thus the above integral equals
\[ \abs{t}^{\frac12}\nu_\psi
\iint_{F^2}(\Phi*(W_s))(\left(\begin{smallmatrix}&&1&\\&&&1\\{-1}&&&\\&-1&&\end{smallmatrix}\right)
\left(\begin{smallmatrix}1&&&\\&t&tz&\\&&t^{-1}&\\&&&1\end{smallmatrix}\right)
\left(\begin{smallmatrix}1&&x&y\\&1&&x\\&&1&\\&&&1\end{smallmatrix}\right))
\psi(\frac12 y)^{-1}\,dx\,dy.
\]
The claim follows.
\end{proof}

\subsection{}\label{sec: localdescentn=1}
For simplicity for $W\in \Ind(\WhitML(\pi))$ we set: (note that $\spclt=\sm{1}{}{}{-1}$ when $n=1$)
\[
M^*W:=(M(\frac12)W)_{-\frac12}
=\iiint_{F^3} W_{\frac12}(\left(\begin{smallmatrix}1&&&\\&-1&&\\&&-1&\\&&&1\end{smallmatrix}\right)
\left(\begin{smallmatrix}&&1&\\&&&1\\-1&&&\\&-1&&\end{smallmatrix}\right)
\left(\begin{smallmatrix}1&&&\\&1&&\\x&y&1&\\z&x&&1\end{smallmatrix}\right)\cdot)\, dx\,dy\,dz.
\]

\begin{claim}\label{claim: phi}
\begin{enumerate}
\item \label{claim: whitfunct}
For any $W\in\Ind(\WhitML(\pi))$ we have $\whitform(W,\Phi,e,s)=\Mint(\Phi *(W_s))$ where
\begin{equation}\label{eq: defmintone}
\Mint(W):=
\int_F W (\left(\begin{smallmatrix}1&&&\\&1&&\\&x&1&\\&&&1\end{smallmatrix}\right)
\left(\begin{smallmatrix}&1&&\\&&&1\\-1&&&\\&&1&\end{smallmatrix}\right)
\left(\begin{smallmatrix}1&-1&&\\&1&&\\&&1&1\\&&&{1}\end{smallmatrix}\right) )\psi(\frac12 x)\, dx.
\end{equation}

\item \label{claim: mainleft}
Let $\tilde W=\whitformd(M(\frac12)\alt{W},\d{\Phi},\cdot,-\frac12)$ for some
$\alt{W}\in \Ind(\WhitMLd(\pi))$ and $\d{\Phi}\in \swrz(F)$.
Then the left-hand side of \eqref{eq: main} equals
\begin{equation} \label{one: beforefe}
\int_{F^*} I^{\psi}(\Phi *(W_{\frac12});\sm{t}{}{}{1})I^{\psi^{-1}}(\d{\Phi}*M^*\alt{W};\sm{t}{}{}{1})\abs{t}^{-3}\,dt
\end{equation}
where for any function $f\in C^{\infty}(\Sp_2)$ we set
\[
I^{\psi}(f;g):=\iiint_{F^3} f (\sm{g}{}{}{g^*}
\left(\begin{smallmatrix}1&&&\\&1&&\\x&y&1&\\z&x&&1\end{smallmatrix}\right)
\left(\begin{smallmatrix}&&1&\\&&&1\\-1&&&\\&-1&&\end{smallmatrix}\right) )\\\psi(\frac{y-z}{2})\ dx\ dy\ dz.
\]
\end{enumerate}
\end{claim}

\begin{proof}[``Proof"]
For part \ref{claim: whitfunct} of the ``Claim", recall that for $g\in \operatorname{SL}_2$,
$\whitform(W,\Phi,\animg{g} ,s)$ is
\[
\iint_{F^2} W_s(\left(\begin{smallmatrix}&1&&\\&&&1\\-1&&&\\&&1&\end{smallmatrix}\right)
\left(\begin{smallmatrix}1&y&&x\\&1&&\\&&1&-y\\&&&1\end{smallmatrix}\right) \left(\begin{smallmatrix}1&&\\&g&\\&&1\end{smallmatrix}\right))
\wevinv(\left(\begin{smallmatrix}1&y&&x\\&1&&\\&&1&-y\\&&&1\end{smallmatrix}\right)\animg{g})
\Phi(1)\,dy\ dx.
\]
For $g=e$, using \eqref{eq: weilH1} and \eqref{eq: weilH3} we get
\begin{multline} \label{eq: whitformat1}
\whitform(W,\Phi,e,s)=
\iint_{F^2}W_s(\left(\begin{smallmatrix}&1&&\\&&&1\\-1&&&\\&&1&\end{smallmatrix}\right)
\left(\begin{smallmatrix}1&y&&x\\&1&&\\&&1&-y\\&&&1\end{smallmatrix}\right))
\Phi(1+y)\,dy\ \psi(\frac12 x)^{-1}\ dx\\=
\int_F \Phi*(W_s)(\left(\begin{smallmatrix}1&&&\\&1&&\\&x&1&\\&&&1\end{smallmatrix}\right)
\left(\begin{smallmatrix}&1&&\\&&&1\\-1&&&\\&&1&\end{smallmatrix}\right)
\left(\begin{smallmatrix}1&-1&&\\&1&&\\&&1&1\\&&&{1}\end{smallmatrix}\right))\ \psi(\frac12 x)\, dx,
\end{multline}
which is $\Mint(\Phi *(W_s))$.

Consider now the left-hand side of \eqref{eq: main}.
Using Bruhat decomposition, we get from \eqref{eq: localinner} that
\[
\tilde{J}(\tilde W,W,\Phi,s)=\int_{F^*}\int_F\tilde W(\widetilde{\sm{}1{-1}{}} \widetilde{\sm t{tz}{}{t^{-1}}})
\whitform(W,\Phi,\widetilde{\sm{}1{-1}{}} \widetilde{\sm t{tz}{}{t^{-1}}},s)\abs{t}^2\ dz\ dt.
\]
Thus, the left-hand side of \eqref{eq: main} equals
\[
\int_F\int_{F^*}\int_F\tilde W(\widetilde{\sm{}1{-1}{}}
\widetilde{\sm t{tz}{}{t^{-1}}}\widetilde{\sm{1}{z'}{}{1}})
\whitform(W,\Phi,\widetilde{\sm{}1{-1}{}} \widetilde{\sm
t{tz}{}{t^{-1}}},\frac12)\abs{t}^2\ dz\ dt\ \psi(-\frac12 z')\ dz',
\]
which by a change of variables $z'\mapsto z'-z$ becomes
\[
\int_{F^*}\big(\int_F\tilde W(\widetilde{\sm{}1{-1}{}}
\widetilde{\sm t{tz'}{}{t^{-1}}})\psi(-\frac12 z')\ dz'\big)\big(
\int_F\whitform(W,\Phi,\widetilde{\sm{}1{-1}{}} \widetilde{\sm
t{tz}{}{t^{-1}}},\frac12)\psi(\frac12 z)\ dz\big)\abs{t}^2\ dt.
\]
By our choice of $\tilde W$, the above integration becomes
\begin{multline*}
\int_{F^*}\big(\int_F
\whitformd(M(\frac12)\alt{W},\d{\Phi},\widetilde{\sm{}1{-1}{}} \widetilde{\sm t{tz'}{}{t^{-1}}},-\frac12)\psi(-\frac12 z')\ dz'\big)\\
\big(\int_F\whitform(W,\Phi,\widetilde{\sm{}1{-1}{}}
\widetilde{\sm t{tz}{}{t^{-1}}},\frac12)\psi(\frac12 z)\
dz\big)\abs{t}^2\ dt.
\end{multline*}

Using \eqref{eq: whitn1} twice, we get that the above integration equals
\begin{multline*}
\int_{F^*} \big(\iiint_{F^3}
(\d{\Phi}*M^*\alt{W})(\left(\begin{smallmatrix}t
&&&\\&1&&\\&&1&\\&&&t^{-1}\end{smallmatrix}\right)
\left(\begin{smallmatrix}1&&&\\&1&&\\x&y&1&\\z&x&&1\end{smallmatrix}\right)
\left(\begin{smallmatrix}&&1&\\&&&1\\-1&&&\\&-1&&\end{smallmatrix}\right))\psi(\frac{z-y}{2})\ dx\ dy\ dz\big)\\
\big(\iiint_{F^3}
(\Phi*(W_{\frac12}))(\left(\begin{smallmatrix}t
&&&\\&1&&\\&&1&\\&&&t^{-1}\end{smallmatrix}\right)
\left(\begin{smallmatrix}1&&&\\&1&&\\x'&y'&1&\\z'&x'&&1\end{smallmatrix}\right)
\left(\begin{smallmatrix}&&1&\\&&&1\\-1&&&\\&-1&&\end{smallmatrix}\right))\psi(\frac{y'-z'}{2})\
dx'\ dy'\ dz'\big) \abs{t}^{-3}\ dt.
\end{multline*}
Here the roots of unity $\nu_{\psi}$ and $\nu_{\psi^{-1}}$ cancel each other. This expression is just \eqref{one: beforefe}.
\end{proof}

\subsection{}\label{sec: ridphin=1}
From ``Claim"~\ref{claim: phi}, we are left to show the identity
\[
\int_{F^*} I^{\psi}(\Phi *(W_{\frac12})
;\sm{t}{}{}{1})I^{\psi^{-1}}(\d{\Phi}*
M^*\alt{W};\sm{t}{}{}{1})\abs{t}^{-3}\,dt=\fel \Mint(\Phi*M^*W)\Mintd(\d{\Phi}*M^*\alt{W}).
\]
Since $\Phi* M^*W=M^*(W')$ with $W'_{\frac12}=\Phi *(W_{\frac12})$, it remains to show that
\begin{equation}\label{goal: n=1}
\int_{F^*} I^{\psi}(W_{\frac12};\sm{t}{}{}{1})I^{\psi^{-1}}(M^*\alt{W};\sm{t}{}{}{1})\abs{t}^{-3}\,dt=
\fel \Mint(M^*W)\Mintd(M^*\alt{W})
\end{equation}
for any $(W,\alt{W})\in \Ind(\WhitML(\pi))\times\Ind(\WhitMLd(\pi))$.

\subsection{} \label{sec: applyFEn=1}
The left-hand side of \eqref{goal: n=1} is a seven-dimensional integral.
To compute it, it is advantageous to integrate over $t$ first.
This is because for any $g\in \Sp_2$ the function
\[
W_s(\left(\begin{smallmatrix}m&\\&m^*\end{smallmatrix}\right)g)\abs{\det m}^{-(s+\frac32)}
\]
belongs to $\WhitM(\pi)$ and therefore the integral over $t$ (with the rest of the variables fixed)
has the form
\[
\int_{F^*} W^1(\sm{t}{}{}{1})W^2(\sm{t}{}{}{1})\ dt
\]
where $W^1\in\WhitM(\pi)$ and $W^2\in\WhitMd(\d\pi)$. The key
observation is that the above integral defines a $\GL_2$-invariant
bilinear form on $\WhitM(\pi)\times\WhitMd(\d\pi)$, and thus
\begin{equation}\label{eq: fen=1}
\int_{F^*} W^1(\sm{t}{}{}{1})W^2(\sm{t}{}{}{1})\ dt
=\int_{F^*} W^1(\sm{t}{}{}{1}b) W^2(\sm{t}{}{}{1}b)\ dt
\end{equation}
for any $b\in \GL_2$.
We conclude that the left-hand side of \eqref{goal: n=1} equals
\[
\int_{F^*} I^{\psi}(W_{\frac12};\sm{t}{}{}{1}b)I^{\psi^{-1}}(M^*\alt{W};\sm{t}{}{}{1}b)\abs{t}^{-3}\abs{\det b}^{-3}\,dt
\]
for any $b\in \GL_2$.

\subsection{}\label{sec: chooseb}
To prove \eqref{goal: n=1} we will show the following:

\begin{claim}\label{claim: main}
Let $b=\sm{-\frac12}{\frac12}{1}{1}=\sm{1}{-\frac12}{}{1}\sm{1}{}{1}{1}\sm{}{1}{1}{}$.
Then for any $W\in\Ind(\WhitML(\pi))$ we have
\begin{equation}\label{one: goal}
I^{\psi}(M^*W;\sm{t}{}{}{1}b)=\abs{t}\fel \Mint(M^*W)
\end{equation}
for all $t\in F^*$   and
\begin{equation}\label{one: factor}
\int_{F^*}I^\psi(W_\frac12;\sm{t}{}{}{1}b)\abs{t}^{-2}\ dt=\Mint(M^*W).
\end{equation}
\end{claim}

In the remaining part of the section we give a formal argument for the ``Claim".
With our choice of $b$, $I^{\psi}(f;\sm{t}{}{}{1}b)$ is equal to
\[
\iiint_{F^3} f(\left(\begin{smallmatrix}t &&&\\&1&&\\&&1&\\&&&t^{-1}\end{smallmatrix}\right)
\left(\begin{smallmatrix}1&&&\\&1&&\\x'&y'&1&\\z'&x'&&1\end{smallmatrix}\right)
\sm b{}{}{b^*}\left(\begin{smallmatrix}&&1&\\&&&1\\-1&&&\\&-1&&\end{smallmatrix}\right))
\\\psi(\frac{y-z}{2})\ dx\ dy\ dz
\]
where $x'=\frac{y-z}2$, $z'=z+y-2x$ and $y'=\frac{z+y+2x}4$.
Making a change of variables, we get
\begin{equation} \label{eq: I0afterb}
I^{\psi}(f;\sm{t}{}{}{1}b)=
\iiint_{F^3} f(\left(\begin{smallmatrix}t &&&\\&1&&\\&&1&\\&&&t^{-1}\end{smallmatrix}\right)
\left(\begin{smallmatrix}1&&&\\&1&&\\ x&y&1&\\z &x&&1\end{smallmatrix}\right)\sm b{}{}{b^*}
\left(\begin{smallmatrix}&&1&\\&&&1\\-1&&&\\&-1&&\end{smallmatrix}\right))\\\psi(x)\ dx\ dy\ dz.
\end{equation}

\subsection{}

We use the fact that the image of $M^*$ admits a nontrivial $H$-invariant linear form,
where $H$ is the centralizer of $\left(\begin{smallmatrix}1&&&\\&-1&&\\&&-1&\\&&&1\end{smallmatrix}\right)$
in $\Sp_2$ i.e.
\[
H=\{\left(\begin{smallmatrix}*&&&*\\&*&*&\\&*&*&\\ *&&&*\end{smallmatrix}\right)\}\cong\SL_2\times\SL_2.
\]
Let $Y$ be the subgroup
\[
\{\left(\begin{smallmatrix}t&&&x\\&t&y&\\&&t^{-1}&\\ &&&t^{-1}\end{smallmatrix}\right):t\in F^*, x,y\in F\}.
\]
Note that in these coordinates the modulus function $\modulus_Y$ of $Y$ is $\abs{t}^4=\abs{\det\sm t{}{}t}^{\frac32+\frac12}$.

For any $W\in\Ind(\WhitML(\pi))$ we have $W_\frac12(yh)=\modulus_Y(y)W_\frac12(h)$. Thus, we can define
\begin{equation}\label{eq: defLWn=1}
L_W(g)=\int_{Y\bs H}W_\frac12(hg)\ dh\\=
\int_{F^*}\iint_{F^2}W_\frac12 (
\left(\begin{smallmatrix}t&&&\\&1&&\\&&1&\\&&&t^{-1}\end{smallmatrix}\right)
\left(\begin{smallmatrix}1&&&\\&1&&\\&x&1&\\z&&&1\end{smallmatrix}\right)g )\ dz\ dx\abs{t}^{-2}\ dt
\end{equation}
and
\[
L'_W(g)=L_W(\left(\begin{smallmatrix}&1&&\\1&&&\\&&&1\\&&1&\end{smallmatrix}\right)g).
\]
Clearly, $L_W$ and $L'_W$ are left $H$-invariant functions. In fact,
\begin{equation} \label{eq: epsilon12}
L'_W=\fel L_W
\end{equation}
where $\fel=\epsilon(\frac12,\pi,\psi)$ because of the local functional equation for $\GL(2)$ (see \S\ref{S:metan=1}).

We can express $M^*W$ in terms of $L_W$ as follows.

\begin{claim} \label{claim: M12}
For any $W\in\Ind(\WhitML(\pi))$,
\[
M^*W(g)=\iint_{F^2} L'_W(\left(\begin{smallmatrix}1&s&r&\\&1&&r\\
&&1&-s\\&&&1\end{smallmatrix}\right)g )\psi(-s)\ dr\ ds.
\]
\end{claim}

\begin{proof}[``Proof"]
It is enough to consider the case $g=e$. The right-hand side is
\begin{multline*}
\iint_{F^2}L'_W(
\left(\begin{smallmatrix}&&&1\\&&1&\\&-1&&\\-1&&&\end{smallmatrix}\right)
\left(\begin{smallmatrix}1&s&r&\\&1&&r\\&&1&-s\\&&&1\end{smallmatrix}\right))\psi(-s)\ dr\ ds=\int_{F^*}\iiiint_{F^4}\\W_{\frac12}(
\left(\begin{smallmatrix}t&&&\\&1&&\\&&1&\\&&&t^{-1}\end{smallmatrix}\right)
\left(\begin{smallmatrix}1&&&\\&1&&\\&x&1&\\z&&&1\end{smallmatrix}\right)
\left(\begin{smallmatrix}&&1&\\&&&1\\-1&&&\\&-1&&\end{smallmatrix}\right)
\left(\begin{smallmatrix}1&s&r&\\&1&&r\\&&1&-s\\&&&1\end{smallmatrix}\right))\psi(-s)\ dr\ ds\ dz\ dx\abs{t}^{-2}\ dt.
\end{multline*}
Making the change of variables $z\mapsto z-rs$, $r\mapsto -r$ we get
\[
\int_{F^*}\iiiint_{F^4}W_{\frac12}(
\left(\begin{smallmatrix}t&&&\\&1&&\\&&1&\\&&&t^{-1}\end{smallmatrix}\right)
\left(\begin{smallmatrix}1&&&\\&1&&\\r&x&1&\\z&r&&1\end{smallmatrix}\right)
\left(\begin{smallmatrix}1&-s&&\\&1&&\\&&1&s\\&&&1\end{smallmatrix}\right)
\left(\begin{smallmatrix}&&1&\\&&&1\\-1&&&\\&-1&&\end{smallmatrix}\right)
)\psi(-s)\ dr\ ds\ dz\ dx\abs{t}^{-2}\ dt.
\]
By a further change of variables in $x$ and $r$ we get
\begin{multline*}
\int_{F^*}\iiiint_{F^4}W_{\frac12}(
\left(\begin{smallmatrix}t&&&\\&1&&\\&&1&\\&&&t^{-1}\end{smallmatrix}\right)
\left(\begin{smallmatrix}1&-s&&\\&1&&\\&&1&s\\&&&1\end{smallmatrix}\right)
\left(\begin{smallmatrix}1&&&\\&1&&\\r&x&1&\\z&r&&1\end{smallmatrix}\right)
\left(\begin{smallmatrix}&&1&\\&&&1\\-1&&&\\&-1&&\end{smallmatrix}\right)
)\psi(-s)\ dr\ ds\ dz\ dx\abs{t}^{-2}\ dt\\=
\iiiint_{F^4}\int_{F^*}W_{\frac12}(
\left(\begin{smallmatrix}t&&&\\&1&&\\&&1&\\&&&t^{-1}\end{smallmatrix}\right)
\left(\begin{smallmatrix}1&&&\\&1&&\\r&x&1&\\z&r&&1\end{smallmatrix}\right)
\left(\begin{smallmatrix}&&1&\\&&&1\\-1&&&\\&-1&&\end{smallmatrix}\right)
)\psi(-s-ts)\abs{t}^{-2}\ dt\ ds\ dr\ dz\ dx.
\end{multline*}
By Fourier inversion this is
\[
\iiint_{F^3}W_\frac12 (
\left(\begin{smallmatrix}-1&&&\\&1&&\\&&1&\\&&&-1\end{smallmatrix}\right)
\left(\begin{smallmatrix}1&&&\\&1&&\\r&x&1&\\z&r&&1\end{smallmatrix}\right)
\left(\begin{smallmatrix}&&1&\\&&&1\\-1&&&\\&-1&&\end{smallmatrix}\right)
 )\ dr\ dz\ dx=M^*W(e)
\]
as required.
\end{proof}

\subsection{}
Equation \eqref{one: goal} follows from \eqref{eq: epsilon12} and the following:
\begin{claim} \label{one: L1}
For all $t\in F^*$ and $W\in\Ind(\WhitML(\pi))$ we have
\begin{subequations}
\begin{gather}
\label{eq: finalIb}
I^{\psi}(M^*W;\sm{t}{}{}{1}b)=
\abs{t}\int_F L_W' (\left(\begin{smallmatrix}1&&&\\&1&&\\ x&&1&\\ &x&&1\end{smallmatrix}\right)
\sm b{}{}{b^*}\left(\begin{smallmatrix}&&1&\\&&&1\\-1&&&\\&-1&&\end{smallmatrix}\right) )\ \psi(x)\ dx,\\
\label{eq: finalmint}
\Mint(M^*W)=\int_F L_W(\left(\begin{smallmatrix}1&& &\\&1&& \\
x&&1&\\&x&&1\end{smallmatrix}\right)
\sm b{}{}{b^*}
\left(\begin{smallmatrix}&&1&\\&&&1\\-1&&&\\&-1&&\end{smallmatrix}\right))\psi(x) \ dx.
\end{gather}
\end{subequations}
\end{claim}

\begin{proof}[``Proof"]
By \eqref{eq: I0afterb} we can write the identity \eqref{eq: finalIb} (for a translate of $W$) as
\begin{multline} \label{one: eq1}
\abs{t}^{-1}\iiint_{F^3}
M^*W (\left(\begin{smallmatrix}t &&&\\&1&&\\&&1&\\&&&t^{-1}\end{smallmatrix}\right)
\left(\begin{smallmatrix}1&&&\\&1&&\\ x&y&1&\\z &x&&1\end{smallmatrix}\right) )
\psi(x)\ dx\ dy\ dz
\\=\int_F L_W' (
\left(\begin{smallmatrix}1&&&\\&1&&\\ x&&1&\\ &x&&1\end{smallmatrix}\right)
 )\psi(x)\ dx.
\end{multline}
By ``Claim" \ref{claim: M12} the left-hand side of \eqref{one: eq1} is
\[
\abs{t}^{-1}\idotsint_{F^5} L_W' (\left(\begin{smallmatrix}1&s&r&\\&1&&r\\
&&1&-s\\&&&1\end{smallmatrix}\right)\left(\begin{smallmatrix}t &&&\\&1&&\\&&1&\\&&&t^{-1}\end{smallmatrix}\right)
\left(\begin{smallmatrix}1&&&\\&1&&\\ x&y&1&\\z &x&&1\end{smallmatrix}\right) )\psi(-s+x)\ dx\ dy\ dz\ dr\ ds.
\]
Conjugating $t$ and making a change of variable in $s$ we get
\[
\idotsint_{F^5} L_W' (\left(\begin{smallmatrix}1&s&r/t&\\&1&&r/t\\&&1&-s\\&&&1\end{smallmatrix}\right)
\left(\begin{smallmatrix}1&&&\\&1&&\\ x&y&1&\\z &x&&1\end{smallmatrix}\right) )\psi(-ts+x)\ dx\ dy\ dz\ dr\ ds.
\]
Conjugating $y$, and changing $s$ to $s-yr/t$ we get
\[
\idotsint_{F^5} L_W' (\left(\begin{smallmatrix}1&s&r/t&\\&1&&r/t\\
&&1&-s\\&&&1\end{smallmatrix}\right)
\left(\begin{smallmatrix}1&&&\\&1&&\\ x&&1&\\z &x&&1\end{smallmatrix}\right) )\psi(-ts+yr+x)\ dx\ dy\ dz\ dr\ ds.
\]
Using Fourier inversion for the integrations over $r$ then $y$, we get
\[
\iiint_{F^3} L_W' (\left(\begin{smallmatrix}1&s& &\\&1&& \\ &&1&-s\\&&&1\end{smallmatrix}\right)
\left(\begin{smallmatrix}1&&&\\&1&&\\ x&&1&\\z &x&&1\end{smallmatrix}\right) )\psi(-ts+x)\ dx \ dz \ ds.
\]
Conjugating $\left(\begin{smallmatrix}1&s& &\\&1&& \\ &&1&-s\\&&&1\end{smallmatrix}\right)$ and using the invariance
of $L_W'$ we get
\[
\iiint_{F^3} L_W' (
\left(\begin{smallmatrix}1&&&\\&1&&\\ x-sz&&1&\\ &x-sz&&1\end{smallmatrix}\right)
\left(\begin{smallmatrix}1&s& &\\&1&& \\&&1&-s\\&&&1\end{smallmatrix}\right) )\psi(-ts+x)\ dx \ dz \ ds
\]
which after a change of variables $x\mapsto x+sz$ becomes
\[
\iiint_{F^3} L_W' (
\left(\begin{smallmatrix}1&&&\\&1&&\\ x&&1&\\ &x&&1\end{smallmatrix}\right)
\left(\begin{smallmatrix}1&s& &\\&1&& \\&&1&-s\\&&&1\end{smallmatrix}\right) )\psi(-ts+x+sz)\ dx \ dz \ ds.
\]
Now integrate over $s$ first and then over $z$. By Fourier inversion we get
\[
\int_F L_W' (\left(\begin{smallmatrix}1&&&\\&1&&\\ x&&1&\\ &x&&1\end{smallmatrix}\right) )\psi(x)\ dx
\]
which is the right-hand side of \eqref{one: eq1}. Note that the last expression is independent of $t$.

Next we show \eqref{eq: finalmint}. We first prove the following identity:
\begin{equation}\label{eq: mintone}
\int_FM^*W (
\left(\begin{smallmatrix}1&&&\\&1&&\\&x&1&\\&&&1\end{smallmatrix}\right))\psi(\frac12 x)\ dx
=\int_F L_W'(\left(\begin{smallmatrix}1&x& &\\&1&& \\
&&1&-x\\&&&1\end{smallmatrix}\right)
\left(\begin{smallmatrix}1&&-\frac12 &\\&1&&-\frac12 \\
&&1&\\&&&1\end{smallmatrix}\right))\psi(-x) \ dx .
\end{equation}

Using ``Claim"~\ref{claim: M12} once again, we can write the left-hand side of \eqref{eq: mintone} as
\[
\iiint_{F^3} L_W' (\left(\begin{smallmatrix}1&s&r&\\&1&&r\\
&&1&-s\\&&&1\end{smallmatrix}\right)
\left(\begin{smallmatrix}1&&&\\&1&&\\&x&1&\\&&&1\end{smallmatrix}\right) )\psi(-s+\frac12 x)\ dr\ ds\ dx.
\]
Conjugating $\left(\begin{smallmatrix}1&&&\\&1&&\\&x&1&\\&&&1\end{smallmatrix}\right)\in H$, this becomes
\[
\iiint_{F^3} L_W' (\left(\begin{smallmatrix}1&s+rx&r&\\&1&&r\\
&&1&-s-rx\\&&&1\end{smallmatrix}\right) )\psi(-s+\frac12 x)\ dr\ ds\ dx
\]
which after a change of variables $s\mapsto s-rx$ becomes
\[
\iiint_{F^3} L_W' (\left(\begin{smallmatrix}1&s&r &\\&1&&r \\
&&1&-s\\&&&1\end{smallmatrix}\right) )\psi(-s+rx+\frac12 x)\ dr\ ds\ dx.
\]
Integrating $r$ first and then $x$ and using Fourier inversion, we get
\[
\int_F L_W' (\left(\begin{smallmatrix}1&s&-\frac12 &\\&1&&-\frac12 \\
&&1&-s\\&&&1\end{smallmatrix}\right) )\psi(-s) \ ds=
\int_F L_W'(\left(\begin{smallmatrix}1&s& &\\&1&& \\
&&1&-s\\&&&1\end{smallmatrix}\right)
\left(\begin{smallmatrix}1&&-\frac12 &\\&1&&-\frac12 \\
&&1&\\&&&1\end{smallmatrix}\right))\psi(-s) \ ds
\]
by the $H$-invariance of $L'_W$.

From \eqref{eq: defmintone} and \eqref{eq: mintone}, we  have
\[
\Mint(M^*W)=
\int_F L_W'(\left(\begin{smallmatrix}1&x& &\\&1&& \\&&1&-x\\&&&1\end{smallmatrix}\right)
\left(\begin{smallmatrix}1&&-\frac12 &\\&1&&-\frac12 \\&&1&\\&&&1\end{smallmatrix}\right)
\left(\begin{smallmatrix}&1&&\\&&&1\\-1&&&\\&&1&\end{smallmatrix}\right)
\left(\begin{smallmatrix}1&-1&&\\&1&&\\&&1&1\\&&&{1}\end{smallmatrix}\right) )\psi(-x) \ dx.
\]
By the invariance of $L_W'$ by $\left(\begin{smallmatrix}&&&1\\&1&&\\&&1&\\-1&&&\end{smallmatrix}\right)\in H$ we get
after a change of variable $x\mapsto -x$:
\begin{multline*}
\int_F L_W'(\left(\begin{smallmatrix}1&& &\\&1&& \\x&&1&\\&x&&1\end{smallmatrix}\right)
\left(\begin{smallmatrix}&&&1\\&1&&\\&&1&\\-1&&&\end{smallmatrix}\right)
\left(\begin{smallmatrix}1&&-\frac12 &\\&1&&-\frac12 \\&&1&\\&&&1\end{smallmatrix}\right)
\left(\begin{smallmatrix}&1&&\\&&&1\\-1&&&\\&&1&\end{smallmatrix}\right)
\left(\begin{smallmatrix}1&-1&&\\&1&&\\&&1&1\\&&&{1}\end{smallmatrix}\right) )\psi(x) \ dx\\=
\int_F L_W'(\left(\begin{smallmatrix}1&& &\\&1&& \\
x&&1&\\&x&&1\end{smallmatrix}\right)\left(\begin{smallmatrix}&1&&\\1&&&\\&&&1\\&&1&\end{smallmatrix}\right)
\sm b{}{}{b^*}
\left(\begin{smallmatrix}&&1&\\&&&1\\-1&&&\\&-1&&\end{smallmatrix}\right))\psi(x) \ dx
\\=
\int_F L_W(\left(\begin{smallmatrix}1&& &\\&1&& \\
x&&1&\\&x&&1\end{smallmatrix}\right)
\sm b{}{}{b^*}
\left(\begin{smallmatrix}&&1&\\&&&1\\-1&&&\\&-1&&\end{smallmatrix}\right))\psi(x) \ dx
\end{multline*}
as claimed.
\end{proof}

\subsection{}

Equation \eqref{one: factor} follows from \eqref{eq: finalmint} and the following:
\begin{claim}\label{claim: factorone}
\[\int_{F^*}I^\psi(W_\frac12;\sm{t}{}{}{1}b)\abs{t}^{-2}\ dt=\int L_W (\left(\begin{smallmatrix}1&&&\\&1&&\\ x&&1&\\ &x&&1\end{smallmatrix}\right)
\sm b{}{}{b^*}
\left(\begin{smallmatrix}&&1&\\&&&1\\-1&&&\\&-1&&\end{smallmatrix}\right)
 )\psi(x)\ dx.\]
\end{claim}
\begin{proof}[``Proof"]
From \eqref{eq: I0afterb}, the left-hand side is
\[
\int_{F^*}\iiint_{F^3} W_\frac12 (\left(\begin{smallmatrix}t &&&\\&1&&\\&&1&\\&&&t^{-1}\end{smallmatrix}\right)
\left(\begin{smallmatrix}1&&&\\&1&&\\ x&y&1&\\z &x&&1\end{smallmatrix}\right)
\sm b{}{}{b^*}
\left(\begin{smallmatrix}&&1&\\&&&1\\-1&&&\\&-1&&\end{smallmatrix}\right)
 )\psi(x)\abs{t}^{-2}\ dx\ dy\ dz\ dt
\]
which by the definition \eqref{eq: defLWn=1} of $L_W$ equals the right-hand side.
\end{proof}

It seems non-trivial to devise a rigorous proof out of the argument above.
The main problem is that the seven-dimensional integral on the left-hand side of \eqref{goal: n=1}
is not absolutely convergent.
Hence, interchanging the order of integration is delicate (especially the step in \S\ref{sec: applyFEn=1}).
One possible simplification is to take $W$ to be supported in the big cell
(given by the non-vanishing of the left lower $2\times 2$ minor). To determine $c_\pi$
it suffices to consider these $W$, which imposes compact support on some of the variables of
integration. Nonetheless, there are still serious convergence
issues remaining, because we can certainly not assume that
$M(\frac12)\alt{W}$ is supported in the big cell as well.
We will address these issues in a forthcoming paper.
We mention that the pertinent generalization of \eqref{eq: fen=1} to general $n$ was carried out in \cite{LMao4}.

\def\cprime{$'$}
\providecommand{\bysame}{\leavevmode\hbox to3em{\hrulefill}\thinspace}
\providecommand{\MR}{\relax\ifhmode\unskip\space\fi MR }

\end{document}